\documentclass[11pt]{amsart}
\usepackage{enumerate}
\usepackage{amsmath}
\usepackage{amsthm}
\usepackage{amsfonts}
\usepackage{amssymb}
\usepackage[numbers]{natbib}
\usepackage{xcolor}
\usepackage{graphicx}

\usepackage{fullpage}

\def\COMMENT#1{}
\let\COMMENT=\footnote% COMMENT OUT for clean outpu

\newtheorem{question}{Question}

\newtheorem{problem}[question]{Problem}
\newtheorem{conjecture}[question]{Conjecture}
\newtheorem{theorem}[question]{Theorem}

\newtheorem{proposition}[question]{Proposition}
\newtheorem{lemma}[question]{Lemma}
\newtheorem{remark}[question]{Remark}
\newtheorem{claim}[question]{Claim}
\newtheorem{definition}[question]{Definition}
\newtheorem{construction}[question]{Construction}
\numberwithin{question}{section}
\numberwithin{equation}{section}
\title{On an extremal problem for locally sparse multigraphs}
\author{Victor Falgas-Ravry}\thanks{Institutionen f\"or Matematik och Matematisk Statistik, Ume{\aa} Universitet, Sweden.\\  \indent Email: {\tt victor.falgas-ravry@umu.se}. Research supported by VR grant 2016-03488.}
\begin{document}

\begin{abstract}
A multigraph $G$ is an $(s,q)$-graph if every $s$-set of vertices in $G$ supports at most $q$ edges of $G$, counting multiplicities. Mubayi and Terry posed the problem of determining the maximum of the product of the edge-multiplicities in an $(s,q)$-graph on $n$ vertices. We give an asymptotic solution to this problem for the family $(s,q)=(2r, a\binom{2r}{2}+\mathrm{ex}(2r, K_{r+1})-1 )$ with $r, a\in \mathbb{Z}_{\geq 2}$. This greatly generalises previous results on the problem due to Mubayi and Terry and to Day, Treglown and the author, who between them had resolved the special case $r=2$. Our result asymptotically confirms an infinite family of cases in (and overcomes a major obstacle to a resolution of) a conjecture of Day, Treglown and the author.
\noindent \textbf{MSC class:}  05C35; 05C22; 05D99.
\end{abstract}
\maketitle	
\section{Introduction}
\subsection{Problem and results}
In this paper, we study a family of extremal problems for multigraphs that are \emph{locally sparse}, in the sense that for some $s\geq 2$, $s$-sets of vertices cannot support too many edges. Formally, we make the following definition:
\begin{definition}
	Given integers $s\geq 2 $ and $q \geq 0$,
	we say a multigraph $G=(V,w)$ is an \emph{$(s,q)$-graph} if every $s$-set of vertices in $V$ supports at most $q$ edges: $\sum _{xy \in X^{(2)}} w(xy) \leq q$ for every $X \in V^{(s)}$. We say such multigraphs have the \emph{$(s,q)$-property}, and denote by $\mathcal F(n,s,q)$ the collection of all $(s,q)$-graphs on the vertex  set $[n]$.
\end{definition}
In 1963, Erd{\H o}s~\cite{Erdos64} raised the question of determining the maximum number of edges an ordinary graph on $n$ vertices with the $(s,q)$-property could have. In the 1990s, Bondy and Tuza and Kuchenbrod considered a first generalisation of this Erd{\H o}s problem to multigraphs.
\begin{definition}
	Given integers $s\geq 2$ and $q\geq 0$, we define
	\begin{align*}
		\mathrm{ex} _\Sigma (n,s,q):= \max \{ e(G) : G \in \mathcal{F}(n,s,q)\} && \textrm{and}&&
		\mathrm{ex} _\Sigma (s,q):= \lim_{n \rightarrow \infty} \frac{\mathrm{ex} _{\Sigma} \left(n,s,q \right)}{\binom{n}{2}}.
	\end{align*}
\end{definition}
Bondy and Tuza~\cite{BondyTuza97} and Kuchenbrod~\cite{Kuchenbrod99} initiated the study of $\mathrm{ex}_{\Sigma}(n,s,q)$. Their results were vastly extended by F\"uredi and K\"undgen~\cite{FurediKundgen02}, who determined the asymptotics of $\mathrm{ex}_{\Sigma}(n,s,q)$ (i.e. the value of $\mathrm{ex}_{\Sigma}(s,q)$) for all pairs $(s,q)$ with $q$ sufficiently large, and in addition determined the exact value in many cases. We are interested in a different generalisation of Erd{\H o}s's question, raised in recent papers of Mubayi and Terry~\cite{MubayiTerry19, MubayiTerry20} with motivation coming from counting problems and applications of container theory to multigraphs.
\begin{definition}
	Given a multigraph $G=(V,w)$, let $P(G)$ denote the product of the edge-multiplicities in $G$,
	\[P(G):=\prod_{uv\in V^{(2)}}w(uv).\]
\end{definition}
\begin{definition}
	Given integers $s\geq 2$ and $q\geq 0$, we define
	\begin{align*}
		\mathrm{ex} _\Pi (n,s,q):= \max \{ P(G) : G \in \mathcal{F}(n,s,q)\} && \textrm{and}&&
		\mathrm{ex} _\Pi (s,q) := \lim_{n \rightarrow \infty} \Bigl( \mathrm{ex} _{\Pi} \left(n,s,q \right)\Bigr)^{1/ \binom{n}{2}}.
	\end{align*}
\end{definition}
\begin{problem}[Mubayi--Terry problem]\label{problem: Mubayi Terry}
Given a pair of integers $s\geq 2, q\geq 0$, determine $\mathrm{ex} _\Pi (s,q)$.
\end{problem}
One may think of $\mathrm{ex} _\Sigma (s,q)$ as the asymptotically maximal arithmetic mean of edge weights in an $(s,q)$-graph, while $\mathrm{ex}_{\Pi}(s,q)$ is the asymptotically maximal geometric mean. By the AM-GM inequality, it is immediate that 
\begin{align*}
\mathrm{ex}_{\Pi}(s,q)\leq \mathrm{ex}_{\Sigma}(s,q),
\end{align*} 
with equality attained if and only if $q=a\binom{s}{2}$ for some integer $a\geq 1$ (in which case both quantities are equal to $a$). But what happens when $q$ lies strictly between $a\binom{s}{2}$ and $(a+1)\binom{s}{2}$? The Mubayi--Terry problem can be rephrased as asking for the extent to which one may improve on the AM--GM inequality for $(s,q)$-graphs with $(s,q)$ in this range, and thereby the extent to which sum-maximisation and product-maximisation differ for these multigraphs.

Our main result in this paper is a resolution of the Mubayi--Terry problem for $(s,q)= (2r, a\binom{2r}{2}+\mathrm{ex}(2r, K_{r+1})-1)$, $a, r\in \mathbb{Z}_{\geq 2}$, where the Tur\'an number $\mathrm{ex}(n, K_{r+1})$ is the maximum number of edges in a $K_{r+1}$-free graph on $n$ vertices.
\begin{theorem}\label{theorem: d=1 asymptotic}
	For any $a,r\in \mathbb{Z}_{\geq 2}$, we have:
	\[\mathrm{ex}_{\Pi}\left(2r, a\binom{2r}{2}+\mathrm{ex}(2r, K_{r+1})-1\right) =a^{\frac{1-x_{\star}(r,1)}{r-1}}(a+1)^{\frac{r-2+x_{\star}(r,1)}{r-1}},\]
	where $x_{\star}(r,1):=\log\left(\frac{a+1}{a}\right)\Big/\log\left(\frac{(a+1)^r}{(a-1)^{r-1}a}\right)$.
\end{theorem}
Theorem~\ref{theorem: d=1 asymptotic} greatly extends earlier results due to Mubayi and Terry~\cite[Theorem~2.4]{MubayiTerry19} and Day, Falgas--Ravry and Treglown~\cite[Theorem ~3.5]{DayFalgasRavryTreglown20+}, who had  established it in the special cases where $r=2$, $a=2$ and where $r=2$, $a\geq 2$ respectively. Further, Theorem~\ref{theorem: d=1 asymptotic} asymptotically confirms an infinite family of cases in a conjecture of Day, Falgas--Ravry and Treglown (see Conjecture~\ref{conjecture: main conjecture} in Section~\ref{subsection: previous work}), and overcomes one of the two main obstacles towards that conjecture identified by these authors by tackling cases where the extremal multigraphs have a much more complicated structure than had previously been managed (the other obstacle being cases where the extremal multigraphs feature a broader range of edge multiplicities --- see the discussion at the beginning of Section~\ref{section: other questions and conjectures}).

A striking feature of the result, which was noted by Mubayi and Terry, is that (assuming Schanuel's conjecture from number theory) the quantity $x_{\star}(r,1)$ is transcendental for all $r,a\geq 2$ (see~\cite[Proposition A.2]{DayFalgasRavryTreglown20+} for a proof of this fact). In particular, the extremal constructions for Theorem~\ref{theorem: d=1 asymptotic} feature partitions of the vertex-set $[n]$ into a number of parts each of which contains an asymptotically transcendental proportion of the vertices. While not wholly surprising given the product setting of the Mubayi--Terry problem, this is still an unusual feature in extremal combinatorics.

As an application of Theorem~\ref{theorem: d=1 asymptotic}, we also obtain the following Erd{\H o}s--Kleitman--Rothschild-type counting results for multigraphs:
\begin{theorem}\label{theorem: counting}
For all integers $a,r \geq 2$, we have: 
\[\left\vert \mathcal{F}\left(n,2r, (a-1)\binom{2r}{2} +\mathrm{ex}(2r, K_{r+1})-1 \right)\right\vert = \left(a^{\frac{1-x_{\star}(r,1)}{r-1}}\left(a+1\right)^{\frac{r-2+x_{\star}(r,1)}{r-1}}\right)^{(1+o(1))\binom{n}{2}}.\]	
\end{theorem}
Further, using a result from~\cite{DayFalgasRavryTreglown20+}, we can derive some further asymptotic cases of a conjecture of Day, Falgas-Ravry and Treglown from Theorem~\ref{theorem: d=1 asymptotic} --- see Theorem~\ref{theorem: further cases} below.

\subsection{Motivation}
The Mubayi--Terry problem is a natural alternative generalisation of the Erd{\H o}s question on the number edges in graphs with the $(s,q)$-property to multigraphs. In addition, it is the ``right'' generalisation insofar as counting problems are concerned. In their classical 1976 paper, Erd{\H o}s, Kleitman and Rothschild~\cite{ErdosKleitmanRotschild76} established that the number of $K_r$-free graphs on $[n]$ is
\[ 2^{\mathrm{ex}(n, K_r)+o(n^2)}.\]
Since this foundational work, there has been a great interest from the research community in estimating the size and characterizing the typical structure of graphs in monotone properties or hereditary properties. The paradigmatic heuristic guiding work in the area is that both size and typical structure should be determined by the size and structure of `extremal' graphs in the property.

In a spectacular breakthrough in 2015, Balogh, Morris and Samotij~\cite{BaloghMorrisSamotij15} and Saxton and Thomason~\cite{SaxtonThomason15} developed powerful theories of hypergraph containers, which have since had a myriad of applications within extremal combinatorics (for more details, see the ICM survey of Balogh, Morris and Samotij~\cite{BaloghMorrisSamotij18}). Using hypergraph containers, one can make the aforementioned heuristic rigorous: given an extremal result and a supersaturation result for a given hereditary graph property, the theory of container immediately implies a counting result for the number of graphs in that property; further, this implication of container theory holds not just for hereditary properties of graphs but also for hereditary properties for a much larger class of objects  (see for instance~\cite{FalgasRavryOConnellUzzell19,Terry18} for general implications of container theory).

In particular, Mubayi and Terry realised that to prove analogues of the Erd{\H o}s--Kleitman--Rotschild theorem (which is about counting graphs with the $(r, \binom{r}{2}-1)$-property) and estimate the number of $(s,q)$-graphs on $n$ vertices, one must determine not the sum-extremal quantity $\mathrm{ex}_{\Sigma}(s,q)$ determined by F\"uredi and K\"undgen in~\cite{FurediKundgen02}, but rather the product-extremal quantity $\mathrm{ex}_{\Pi}(s,q)$. More precisely, Mubayi and Terry showed in~\cite[Theorem 2.2]{MubayiTerry19} that for $q>\binom{s}{2}$,
\begin{align}\label{eq: MT counting}
\Bigl\vert \mathcal{F}\bigl(n,s,q-\binom{s}{2}\bigr)\Bigr\vert = \mathrm{ex} _\Pi (s,q)^{\binom{n}{2}+o(n^2)}.
\end{align}
Thus the Erd{\H o}s--Kleitman--Rothschild-type problem of estimating the size of the multigraph family $\mathcal{F}(n,s,q-\binom{s}{2})$ is equivalent to the Tur\'an-type  problem of determining $\mathrm{ex} _\Pi (s,q)$. This motivated Mubayi and Terry's introduction of Problem~\ref{problem: Mubayi Terry}, and also shows how Theorem~\ref{theorem: counting} is an immediate corollary of Theorem~\ref{theorem: d=1 asymptotic}.

\subsection{Previous work and extremal constructions}\label{subsection: previous work}
Mubayi and Terry resolved Problem~\ref{problem: Mubayi Terry} in the cases where $a\binom{s}{2}-\frac{s}{2}\leq q\leq a\binom{s}{2}+s-2$, $a\in \mathbb{Z}_{\geq 2}$~\cite[Theorems 3--4]{MubayiTerry20}, $(s,q)=(4, 9)$~\cite[Theorem 5]{MubayiTerry20} and $(s,q)=(4, 15)$~\cite[Theorem 2.4]{MubayiTerry19}. Day, Falgas-Ravry and Treglown~\cite[Theorem 3.5]{DayFalgasRavryTreglown20+} extended the latter result, determining $\mathrm{ex}_{\Pi}(s,q)$ when $(s,q)=(4, a\binom{4}{2}+3)$ with $a\in \mathbb{Z}_{\geq 2}$. In addition the same authors determined $\mathrm{ex}_{\Pi}(s,q)$ when $(s,q)=(s, a\binom{s}{2} + 2s-5)$ with $s\in \{5,6, 7\}$ \cite[Theorems 3.6--3.8]{DayFalgasRavryTreglown20+}, and  when $q=a\binom{s}{2}+\mathrm{ex}(s, K_{r+1})$ with $a\in \mathbb{Z}_{\geq 1}$, $s>r\geq 2$ \cite[Theorem 3.10]{DayFalgasRavryTreglown20+}. Further, they formulated a general conjecture on the value of $\mathrm{ex}_{\Pi}(s,q)$ for many values of $(s,q)$, which guides our work in the present paper.

To state their conjecture precisely, we must first introduce some notation and define certain families of constructions. For $n\in \mathbb{N}$, write $[n]$ as a shorthand for the set $\{1,2,\ldots ,n\}$. Day, Falgas-Ravry and Treglown considered the following family of constructions in~\cite{DayFalgasRavryTreglown20+}:
\begin{construction}\label{construction: lower bound}
	Let $a,r \in \mathbb{N}$ and $d\in \{0\}\cup[a-1]$. Given $n \in \mathbb{N}$, let $\mathcal{T}_{r,d}(a,n)$ denote the collection of multigraphs $G$ on $[n]$ for which $V(G)$ can be partitioned into $r$ parts $V_0, \ldots ,V_{r-1}$ such that:
	\begin{enumerate}[(i)]
		\item all edges in $G[V_0]$ have multiplicity $a-d$;
		\item for all $i\in [r-1]$,  all edges in $G[V_i]$ have multiplicity $a$;
		\item all other edges of $G$ have multiplicity $a+1$.
	\end{enumerate} 
	Given $G\in \mathcal{T}_{r,d}(a,n)$, we refer to $\sqcup_{i=0}^{r-1}V_i$ as the \emph{canonical partition} of $G$.  
\end{construction}	
\begin{definition}
	Let $\Sigma_{r,d}(a,n)$ and $\Pi_{r,d}(a,n)$ respectively denote the maximal edge-sum and the maximal edge-product that can be achieved inside $\mathcal{T}_{r,d}(a,n)$,
	\begin{align*}
	\Sigma_{r,d}(a,n) := \max \{e(G) : G  \in \mathcal{T}_{r,d}(a,n)\}, &&
	\Pi_{r,d}(a,n) := \max \{P(G) : G  \in \mathcal{T}_{r,d}(a,n)\}.
	\end{align*}
%	Let also $T^{S}_{r,d}(a,n)$ denote the family of multigraphs $G \in \mathcal{T}_{r,d}(a,n)$ with $e(G)=\Sigma_{r,d}(a,n)$ and $T^{P}_{r,d}(a,n)$ denote the family of multigraphs $G\in \mathcal{T}_{r,d}(a,n)$ with $P(G) =\Pi_{r,d}(a,n)$.  (Note that these two multigraph families are in general quite different.)
\end{definition}
As shown in~\cite{DayFalgasRavryTreglown20+}, the sum- and product-extremal multigraphs in $\mathcal{T}_{r,d}(a,n)$ have in general very different canonical partitions. Indeed~\cite[(3.1)--(3.3)]{DayFalgasRavryTreglown20+}, in sum-maximising multigraphs $\vert V_0\vert = \frac{1}{d(r-1)+r}n+o(n)$, while in product maximising multigraphs $\vert V_0\vert =x_{\star}(r,d)n+o(n)$, where $x_{\star}$ is the function of $a,r,d$ given by
\begin{align}\label{eq: def of xstar} x_{\star}(r,d):= \frac{\log\left(\frac{a+1}{a}\right)}{\log \left(\frac{(a+1)^r}{(a-d)^{r-1}a}\right)},\end{align}
which is strictly smaller than $\frac{1}{d(r-1)}$~\cite[Proposition 5.4]{DayFalgasRavryTreglown20+}; in both cases, the remaining parts $V_1$ to $V_{r-1}$ have balanced sizes. Substituting in the values of the optimal part sizes for sum-maximisation and product-maximisation respectively, one obtains the following asymptotic expressions for $\Sigma_{r,d}(a,n)$ and $\Pi_{r,d}(a,n)$:
\begin{align}\label{ineq: sigma and pi estimate}
\Sigma_{r,d}(a,n)=\left(a+1-\frac{d+1}{(r-1)d+r}+o(1)\right)\binom{n}{2}, && \Pi_{r,d}(a,n)= \left(a^{\frac{1-x_{\star}(r,d)}{r-1}} (a+1)^{\frac{r-2+x_{\star}(r,d)}{r-1}+o(1)}\right)^{\binom{n}{2}}.
\end{align}
Since multigraphs in $\mathcal{T}_{r,d}(a,n)$ have the $(s,\Sigma_{r,d}(a,s))$-property, it is immediate that
\begin{align}\label{eq: trivial lower bound on expi}
\mathrm{ex}_{\Pi}(n, s, \Sigma_{r,d}(a,s)) \geq \Pi_{r,d}(a,n).
\end{align} Day, Falgas-Ravry and Treglown conjectured~\cite[Conjecture 3.2]{DayFalgasRavryTreglown20+} that this lower bound is tight for $s,n$ sufficiently large, i.e.\ that  $\mathrm{ex}_{\Pi}(n, s, \Sigma_{r,d}(a,s))$ is attained by product-extremal multigraphs from $\mathcal{T}_{r,d}(a,n)$.
\begin{conjecture}[Day, Falgas-Ravry and Treglown~\cite{DayFalgasRavryTreglown20+}]\label{conjecture: main conjecture} 
	For all integers $a,r,s, d$ with $a,r \geq 1$, $d\in \{0\}\cup[a-1]$, $s\geq (r-1)(d+1)+2$ and all $n$ sufficiently large,
	\[ \mathrm{ex}_{\Pi}(n, s, \Sigma_{r,d}(a,s)) = \Pi_{r,d}(a,n).\]
\end{conjecture}
\begin{remark}
Observe that for any $r\geq 2$, $\Sigma_{r,1}(a,2r)=a\binom{2r}{2}+\mathrm{ex}(2r, K_{r+1})-1$, so that Theorem~\ref{theorem: d=1 asymptotic} asymptotically confirms this conjecture when $s=2r$ and $d=1$.	
\end{remark}	
The condition $s\geq (r-1)(d+1)+2$ in Conjecture~\ref{conjecture: main conjecture} above is related to the minimum `size' of $s$ such that sum-extremal $s$-vertex subgraphs can tell $\mathcal{T}_{r,d}(a,n)$ apart from $\mathcal{T}_{r',d'}(a,n)$ with $r<r'$ or $r=r'$ and $d'<d$. In addition to Construction~\ref{construction: lower bound}, the authors of~\cite{DayFalgasRavryTreglown20+} also provided some other families of constructions to bridge some of the gaps and cover $q$-s lying in the intervals between successive values of $\Sigma_{r,d}(a,s)$. These are however significantly more intricate --- they correspond to iterated versions of Construction~\ref{construction: lower bound} --- and do not give a complete picture; see the discussion in Section~\ref{section: other questions and conjectures}.

Extending earlier results of Mubayi and Terry, Day, Falgas-Ravry and Treglown showed Conjecture~\ref{conjecture: main conjecture} is true for
	\begin{itemize}
		\item $d=0$ and all $a\geq 1$, $s>r\geq 1$~\cite[Theorem 3.10]{DayFalgasRavryTreglown20+} --- this can be viewed as a multigraph  generalisation of Tur\'an's theorem, with $q=a\binom{s}{2}+\mathrm{ex}(s, K_{r+1})$;
		\item $d=1,\  r=2,\  s\in \{4,5,6, 7\}$ and all $a\geq 2$~\cite[Theorems 3.5--3.8]{DayFalgasRavryTreglown20+}. 
	\end{itemize}
Further they showed~\cite[Theorem 3.11]{DayFalgasRavryTreglown20+} that for $r,d$ fixed the `base cases' $s=(r-1)(d+1)+2$, $a\geq d+1$ of Conjecture~\ref{conjecture: main conjecture} implies the `higher cases' $s'>(r-1)(d+1)+2$ hold for all $a>a_0=a_0(r,d, s')$ sufficiently large. As an immediate corollary of this latter result and of Theorem~\ref{theorem: d=1 asymptotic} proved in this paper, we have the following:
\begin{theorem}\label{theorem: further cases}
For all integers $r,s$ with $s\geq 2r\geq 4$ and all positive integers $a$ sufficiently large,
\[\mathrm{ex}_{\Pi}(s, \Sigma_{r,1}(a,s))=(a+1)^{\frac{r-2+x_{\star}(r,1)}{r-1}} a^{\frac{1-x_{\star}(r,1)}{r-2}}.\]	
\end{theorem}
\noindent In other words, Conjecture~\ref{conjecture: main conjecture} is asymptotically true for $d=1$ and $a$ sufficiently large.

Finally, it would be remiss not to mention here the results of F\"uredi and K\"undgen. These authors showed in~\cite{FurediKundgen02} that for all $q$ sufficiently large\footnote{Formally F\"uredi and K\"undgen proved the upper bound and a matching lower bound hold for $\mathbb{Z}$-weighted graphs, i.e.\   allowing negative edge multiplicities. However their lower bound construction only involve positive weights when $q$ is sufficiently large with respect to $s$.}
\[\mathrm{ex}_{\Sigma}(s,q)= \min\left\{m\in \mathbb{Q}: \ \sum_{i=1}^{s-1}\left\lfloor 1+m i \right\rfloor >q \right\}.\]
In the particular case $(s,q)=\left(2r, \Sigma_{r,1}(a, 2r)\right)=\left(2r, (a+1)\binom{2r}{2}-r-1\right)$, this gives that for all $a$ sufficiently large
\begin{align*}
\mathrm{ex}_{\Sigma}(2r,\Sigma_{2,1}(a,2r))=a+\frac{2r-3}{2r-1},
\end{align*}
with the same extremal constructions as the ones that attain $\Sigma_{r,1}(a,n)$ inside the family $\mathcal{T}_{r,1}(a,n)$. Thus for both sum-maximisation and production-maximisation of $(2r, \Sigma_{r,1}(a,2r))$-graphs one must look to the generalised Tur\'an multigraphs from Construction~\ref{construction: lower bound}, albeit with different weights assigned to the various parts in the canonical partition.

\subsection{Proof ideas and organisation of the paper}
The proof of Theorem~\ref{theorem: d=1 asymptotic} proceeds by induction on $r$, and by structural and optimisation arguments. The base case $r=2$ was proved in~\cite[Theorem 3.5]{DayFalgasRavryTreglown20+}, using somewhat different arguments. For the inductive step, a simple vertex-removal argument shows it is sufficient to establish that for all $n$ sufficiently large, every $G\in \mathcal{F}\left(n, 2r+2, \Sigma_{r+1,1}(a,2r+2)\right)$ must contain a vertex with low product-degree. Further, one can show that one can restrict one's attention to those multigraphs $G$ belonging to a certain subset of the family $\mathcal{F}\left(n, 2r+2, \Sigma_{r+1,1}(a,2r+2)\right)$ with helpful properties (see Proposition~\ref{prop: degree removal final form}).

 A key observation is then that in a product-extremal multigraph from $\mathcal{T}_{r+1,1}(a,n)$ there are two kinds of vertices: those from $V_i$, $i\in [r]$, sending edges of multiplicity $a+1$ to a set of vertices inducing a copy of a product-extremal multigraph from $\mathcal{T}_{r,1}\left(a,\left(\frac{r-1+x_{\star}(r+1,1)}{r}\right)n\right)$, and those from $V_0$, sending edges of multiplicity $a+1$ to a set of vertices inducing a copy of a product-extremal multigraph from $\mathcal{T}_{r,0}\left(a,\left(1-x_{\star}(r+1,1)\right)n\right)$.

Our arguments builds on a similar dichotomy: we pick a vertex $x$, and consider the set $Y$ of vertices sending edges of multiplicity $a+1$ to $x$.  If the product of edge multiplicities in $Y$ is not much larger than if $Y$ had the $(2r, \Sigma_{r,1}(a, 2r))$ property, then we use optimisation arguments to show that that $G$ contains a vertex with low product-degree (Lemma~\ref{lemma: if a+1 neighbourhoods are sparse, done}). On the other hand, if the product of edge multiplicities in $Y$ is larger than this, then we show by a different argument that we can find certain `good' $r$-partite structures in $Y$, which can then be used to show that $G$ contains a vertex with low product-degree (Lemma~\ref{lemma: if contain good Kr structure, then done}). The bulk of the work of the paper is showing these `good' $r$-partite structures can indeed be found. This requires some careful structural analysis and some weighted geometric averaging arguments that together form the paper's main technical innovations on previous work.

Our paper overcomes one of the two main challenges towards a proof of Conjecture~\ref{conjecture: main conjecture} (the other, harder one being the case $d\geq 2$ where the conjectured extremal examples feature a broader range of edge multiplicities). We suspect large parts of the proof framework we have developed in this paper may be helpful in tackling the general case of Conjecture~\ref{conjecture: main conjecture}, so we have taken some care to present it in a modular fashion, and in particular to prove general forms of our key lemmas.

The paper is structured as follows. In the next subsection, we gather some useful notation. In \textbf{Section~\ref{section: preliminaries, props of near extremal multigraphs}}, we show that in investigations of Conjecture~\ref{conjecture: main conjecture} we can restrict our attention to multigraphs $G$ with a much more favourable structure (Proposition~\ref{prop: degree removal final form}). In \textbf{Section~\ref{section: optimisation in neighbourhoods}} we prove some very general optimisation lemmas, in particular  Lemmas~\ref{lemma: if a+1 neighbourhoods are sparse, done} and~\ref{lemma: if contain good Kr structure, then done} alluded to above. Finally in \textbf{Section~\ref{section: finding r-partite}} we leave the general setting and specialise to the cases $(s,q)= \left(2r+2, \Sigma_{r+1,1}(2r+2)\right)$; we establish the existence of `good' $r$-partite structures (or low product-degree vertices) in this section, completing the proof of Theorem~\ref{theorem: d=1 asymptotic}. We end the paper in \textbf{Section~\ref{section: other questions and conjectures}} with some remarks on further questions, open problems and future directions for work on the Mubayi--Terry problem.

\subsection{Notation} Given a set $A$ and $t\in \mathbb{Z}_{\geq 0}$, we let $A^{(t)}$ denote the collection of all subsets of $A$ of size $t$. A multigraph is a pair $G=(V,w)$, where $V=V(G)$ is a set of vertices and $w=w_G$ is a function $w: \ V^{(2)}\rightarrow \mathbb{Z}_{\geq 0}$ assigning to each pair $\{a,b\}\in V^{(2)}$ a \emph{weight} or \emph{multiplicity} $w_G(\{a,b\})$. We usually write $ab$ for $\{a,b\}$ and, when the host multigraph $G$ is clear from context, we omit the subscript $G$ and write simply $w(ab)$ for $w_G(\{a,b\})$.

Given a multigraph $G$ and a set $X\subseteq V(G)$, we write $S(G[X])$ or, when the host multigraph $G$ is clear from context, $S(X)$ for the sum of the edge multiplicities of $G$ inside $X$, \[S(G[X]):=\sum_{v_1v_2 \in X^{(2)}} w(v_1v_2).\] 
Similarly, we write $P(G[X])$ or $P(X)$ for the product of the edges multiplicities of $G$ inside $X$, 
\[P(G[X]):=\prod_{v_1v_2\in X^{(2)}}w(v_1v_2).\] 
Further, given disjoint sets $X,Y\subset V(G)$ we write $S(G[X,Y])$  ($S(X,Y)$) and $P(G[X,Y])$ ($P(X,Y)$) for, respectively the sum  and the product of the edge multiplicities of $G$ over all edges $xy$ with $x\in X$ and $y\in Y$.

We define $d_G (v)$ (or simply $d(v)$) to be $\sum_{u\in V(G)} w_G(uv)$, and refer to this quantity as the \emph{degree}
of $v$ in $G$. Analogously, we use $p_G(v)$ to denote $\prod_{u\in V(G)} w_G(uv)$, and refer to this quantity as the \emph{product-degree} of $v$ in $G$. When $G$ is clear from context we write $p(v)$ for $p_G(v)$, and given $X\subseteq V(G)$ we also use $p_X(v)$ to denote $P(\{v\}, X)$, the product of the edge multiplicities of the edges sent by $v$ into $X$ in the multigraph $G$.

Finally, in our arguments we will will need to consider the subgraph of edges with a given multiplicity $m$ in a multigraph $G$. It is therefore convenient to introduce the following notation.
\begin{definition}
	Given a multigraph $G=(V,w)$, and $m\in \mathbb{Z}_{\geq 0}$, let $G^{(m)}$ denote the ordinary graph given by
	\[G^{(m)}= \Bigl(V,  \{e\in V(G)^{(2)}:\ w(e)=m  \}\Bigr).\]
	Given $v\in V$ and a subset $X\subseteq V$, we also let
	\[N^{(m)}_X(v):=\{x\in X: \ w(vx)=m\}.\]
	We refer to $G^{(m)}$ as the $m$-subgraph of $G$ and to $N^{(m)}_X(v)$ as the $m$-neighbourhood of $v$ in $X$. When $X=V$, we drop the subscript $X$ and simply write $N^{(m)}(v)$ for the $m$-neighbourhood of $v$ in $G$.
\end{definition}

\section{Preliminaries: properties of near-extremal multigraphs}\label{section: preliminaries, props of near extremal multigraphs}
Our problem involves interaction between sums and products. It is thus unsurprising that an integral version of the AM--GM inequality plays a part in our arguments.
\begin{proposition}[Integral AM--GM inequality]
	% ee e.g. Proposition 5.5in~\cite{DayFalgasRavryTreglown20+}]
	\label{prop: integral AM-GM}
	Let $a,n\in \mathbb{N}$, $t\in  \{0\}\cup[n]$, and let $w_1, \ldots, w_n$ be non-negative integers with $\sum_{i=1}^n w_i= an +t$. Then 
	%the following hold:
%	\begin{enumerate}[(i)]
	$\prod_{i=1}^n w_i \leq a^{n-t}(a+1)^t$. 
	%	\item if $t\leq n-2$ and $w_1=a-1$, then $\prod_{i=1}^n w_i\leq (a-1)a^{n-t-2}(a+1)^{t+1}$.
%	\end{enumerate}
\end{proposition}	
\noindent We shall also repeatedly use the following simple weighted geometric averaging bound.
\begin{proposition}[Weighted geometric averaging]\label{prop: geometric averaging}
Let $\alpha_1, \alpha_2, \ldots \alpha_m$ be non-negative real numbers summing to $1$, and let $p_1, p_2, \ldots, p_m$ be non-negative real numbers. Then there exists some $i\in[m]$ such that $p_i$ is at most the $(\alpha_j)_{j=1}^{m}$-weighted geometric mean of the quantities $(p_j)_{j=1}^m$:
\[p_i \leq \prod_{i=1}^m {(p_i)}^{\alpha_i}.\]	
\end{proposition}

\subsection{Behaviour of $\Sigma_{r,d}(a,s)$ and $\Pi_{r,d}(a,n)$}
It shall be useful in our proof of Theorem~\ref{theorem: d=1 asymptotic} to understand the size and structure of sum-maximising multigraphs from $\mathcal{T}_{r,d}(a,s)$. To this end, we shall use the following proposition from~\cite{DayFalgasRavryTreglown20+}:
\begin{proposition}[Proposition~5.3 from~\cite{DayFalgasRavryTreglown20+}]\label{prop: sum-extremal subgraphs, threshold for 2 vertices in extremal part}
	Let $r\in \mathbb{Z}_{\geq 2}$, $a\in \mathbb{N}$, $d\in \{0\}\cup[a-1]$. Let $s,j\in \mathbb{N}$. Then there exists $G\in \mathcal{T}_{r,d}(a,s)$ with $e(G)=\Sigma_{r,d}(a,s)$ whose canonical partition $\sqcup_{i=0}^{r-1}V_i$ satisfies $\vert V_0\vert =j$ if and only if one of the following hold: 
	\begin{enumerate}[(a)]
		\item $j>1$ and $(r-1)(d+1)(j-1)+(j-1)+r-1\leq s\leq (r-1)(d+1)j +j+r-1$;
		\item $j=1$ and $s\leq (r-1)(d+1)+r=(r-1)(d+2)+1$.
	\end{enumerate}
\end{proposition}
\noindent In particular, Proposition~\ref{prop: sum-extremal subgraphs, threshold for 2 vertices in extremal part} implies that for $s'<(r-1)(d+1)+2$ we have 
\begin{align}\label{eq: bound on Sigma diff}
\Sigma_{r,d}(a, s'+1)-\Sigma_{r,d}(a,s')=s'(a+1)-\Bigl\lfloor\frac{s'-1}{r-1}\Bigr\rfloor\geq s'(a+1)-d-1,
\end{align}
since there exist sum-maximising multigraphs in $\mathcal{T}_{r, d}(a, s'+1)$ and $\mathcal{T}_{r, d}(a, s')$ whose canonical partitions satisfy $\vert V_0\vert = 1$. Also, for $i\in [d]$, Proposition~\ref{prop: sum-extremal subgraphs, threshold for 2 vertices in extremal part} implies
\begin{align}\label{eq: bound on differnence of sigmas d, d-i}
\Sigma_{r,d}(a, (r-1)(d-i+2)+2)< \Sigma_{r,d-i}(a, (r-1)(d-i+2)+2).
\end{align}
Indeed, given $G'\in \mathcal{T}_{r,d}(a, (r-1)(d-i+2)+2)$ with $e(G')=\Sigma_{r,d}(a, (r-1)(d-i+2)+2)$, consider the graph $G\in \mathcal{T}_{r,d-i}(a, (r-1)(d-i+2)+2)$ obtained from $G'$ by replacing each edge with multiplicity $a-d$ by an edge with multiplicity $a-d+i$.

If the canonical partition of $G'$ satisfies $\vert V_0\vert >1$, then clearly $e(G')<e(G)\leq  \Sigma_{r,d-i}(a, (r-1)(d-i+2)+2)$. On the other hand if the canonical partition of $G'$ (which is also a canonical partition of $G$) satisfies $\vert V_0\vert =1$, then by Proposition~\ref{prop: sum-extremal subgraphs, threshold for 2 vertices in extremal part} part (b), the graph $G$ is not sum-extremal in $\mathcal{T}_{r,d-i}(a, (r-1)(d-i+2)+2)$, and thus  $e(G')=e(G)<\Sigma_{r,d-i}(a, (r-1)(d-i+2)+2)$. This establishes~\eqref{eq: bound on differnence of sigmas d, d-i}.

\begin{proposition}[Equation (3.3) in \cite{DayFalgasRavryTreglown20+}]\label{prop: value of xstar}
Let $r\in \mathbb{Z}_{\geq 2}$, $a\in \mathbb{N}$, $d\in \{0\}\cup[a-1]$, $G\in \mathcal{T}_{r,d}(a,n)$ with $P(G)=\Pi_{r,d}(a, n)$, and let $\sqcup_{i=0}^{r-1}V_i$ be the canonical partition of $G$. Then the following hold:
	\begin{enumerate}[(i)]
\item $\vert V_0\vert = x_{\star}(r,d)n+O(1)$;
\item $P(G)=\Pi_{r,d}(a,n)= a^{\binom{n}{2}}\left(\frac{a+1}{a}\right)^{\left(\frac{r-2+x_{\star}(r,d)}{r-1}\right)\binom{n}{2}+O(n)}$,
\end{enumerate}
where we recall that $x_{\star}(r,d)$ was defined in~\eqref{eq: def of xstar} and is given by
$x_{\star}(r,d):=  \frac{\log\left(\frac{a+1}{a}\right)}{\log\left(\frac{(a+1)^r}{a(a-d)^{r-1}}\right)}$.
\end{proposition}
\begin{remark} 
	The quantity $x_\star(r,d)$ satisfies the following recurrence relation: for all $r,d\geq 1$, 
\begin{align}\label{eq: recurrence for xstar}
x_{\star}(r+1,d) =\left(\frac{r-1+x_{\star}(r+1,d)}{r}\right)x_{\star}(r,d).
\end{align} 
This identity can be verified algebraically, and has a natural combinatorial interpretation. Consider a product-maximising multigraph $G$ from $\mathcal{T}_{r+1,d}(a,n)$ and let $\sqcup_{i=0}^{r}V_i$ be its canonical partition. Then the parts $\sqcup_{i=0}^{r-1}V_i$ induce an (almost) product maximising multigraph $G'$ from $\mathcal{T}_{r, d}(a, n-\vert V_r\vert)$. In particular the special part $V_0$ satisfies both $\vert V_0\vert =x_{\star}(r+1, d)n +o(n)$ and $\vert V_0\vert = x_{\star}(r, d)(n -\vert V_r\vert)(1+o(1))= x_{\star}(r,d)\left(\frac{r-1+x_{\star}(r+1,d)}{r}\right)n+o(n)$.

\noindent Note that in the special case $r=1$, equality~\eqref{eq: recurrence for xstar} is true but vacuous: $x_{\star}(1,d)=1$ for all $d$, and~\eqref{eq: recurrence for xstar} is the tautological fact that $x_{\star}(2,d)=x_{\star}(2,d)$.
\end{remark}

\subsection{Properties of near-extremal multigraphs}
Fix positive integers $a,r,d$ with $r\geq 3$, $d\geq 1$, $a\geq d+1$. Set $s= (r-1)(d+1)+2$. Let $x_{\star}(r,d)$ be as in Proposition~\ref{prop: value of xstar}. As we show in the elementary proposition below, one can essentially reduce the problem of showing $\mathrm{ex}_{\Pi}(n,s,q)= \left(\Pi_{r,d}(a,n)\right)^{1+o(1)}$ to the problem of showing all $(s,q)$-graphs contain vertices with low product-degree.
\begin{proposition}\label{prop: degree removal}
	If for all $n\in \mathbb{N}$ we have 
	%there exists a function $f: \ \mathbb{N}\rightarrow \mathbb{R}$ with $f(n)=o(n)$ such that
	\begin{align}\label{eq: degree removal inequality}
	\mathrm{ex}_{\Pi}\left(n+1, s, \Sigma_{r,d}(a,s)\right)\leq \mathrm{ex}_{\Pi}\left(n, s, \Sigma_{r,d}(a,s)\right) a^{n} \left(\frac{a+1}{a}\right)^{\left(\frac{r-2+x_{\star}(r,d)}{r-1}\right)n +o(n)},  \end{align}
	then
	\[\textrm{ex}_{\Pi}(n,s,\Sigma_{r,d}(a,s))= \left(\Pi_{r,d}(a,n)\right)^{1+o(1)}.\]
\end{proposition}
\begin{proof}
Applying our hypothesis $n-1$ times, we have
\begin{align*}
\mathrm{ex}_{\Pi}\left(n, s, \Sigma_{r,d}(a,s)\right)\leq \prod_{i=1}^{n-1}a^{i}\left(\frac{a+1}{a}\right)^{\left(\frac{r-2+x_{\star}(r,d)}{r-1}\right)i + o(i)}=  a^{\binom{n}{2}}\left(\frac{a+1}{a}\right)^{\left(\frac{r-2+x_{\star}(r,d)}{r-1}\right)\binom{n}{2}+o(n^2)},
\end{align*}
and the claim then follows from Proposition~\ref{prop: value of xstar} part (ii) together with our observation in~\eqref{eq: trivial lower bound on expi} that $\Pi_{r,d}(a,n)$ is a lower bound for $\mathrm{ex}_{\Pi}\left(n, s, \Sigma_{r,d}(a,s)\right)$.
\end{proof}
\noindent Our goal in this subsection is to show that  we may in fact restrict our attention to the problem of showing that all multigraphs within a certain `nice' subset of $\mathcal{F}(n,s,\Sigma_{r,d}(a,s))$ contain vertices with low product-degrees.

\begin{definition}
	Let $\mathcal{G}\left(n,s, \Sigma_{r,d}(a,s)\right)$ be the set of multigraphs on $[n]$ that have the $(s', \Sigma_{r,d}(a,s'))$-property for all integers $s'$: $2\leq s'\leq s$.
\end{definition}
\begin{remark}
	Observe that multigraphs in $\mathcal{G}\left(n,s, \Sigma_{r,d}(a,s)\right)$ have bounded multiplicity: they have the $(2, \Sigma_{r,d}(a,2))=(2, a+1)$-property, meaning that in such multigraphs all edges have multiplicity at most $a+1$. We shall make heavy use of this fact in our proof of Theorem~\ref{theorem: d=1 asymptotic}.
\end{remark}
\noindent As a corollary of (the proof of)  a result of Day, Falgas-Ravry and Treglown~\cite{DayFalgasRavryTreglown20+}, every multigraph $G\in\mathcal{F}(n,s,\Sigma_{r,d}(a,s))$ either belongs to the more restricted subfamily $\mathcal{G}\left(n,s, \Sigma_{r,d}(a,s)\right)$ or contains a vertex with low product-degree.
\begin{proposition}[Corollary of {\cite[Theorem 6.1]{DayFalgasRavryTreglown20+}}]\label{prop: good multigraphs}
Let $G \in \mathcal{F}(n,s,\Sigma_{r,d}(a,s))$. Then either there exists $v\in V(G)$ with  
\[p_G(v)\leq a^{n} \left(\frac{a+1}{a}\right)^{\left(\frac{r-2}{r-1}\right)n +O(1)},\]
	or $G\in \mathcal{G}(n,s, \Sigma_{r,d}(a,s))$	
	\end{proposition}
\noindent Generalising ideas from~\cite{DayFalgasRavryTreglown20+, MubayiTerry20}, we consider an even nicer subfamily of $\mathcal{G}(n,s,\Sigma_{r,d}(a,s))$. 
\begin{definition}
	Two vertices $u$ and $v$ in a multigraph $G$ are \emph{clones} if for every $z\in V(G)\setminus\{x,y\}$ we have $w_G(xz)=w_G(yz)$. 
\end{definition}
\begin{definition}
	Let $\mathcal{H}(n, s, \Sigma_{r,d}(a,s))$ denote the set of multigraphs $G$ from $\mathcal{G}(n,s,\Sigma_{r,d}(a,s))$ such that:
	\begin{enumerate}[(i)]
		\item every edge in $G$ has weight at least $a-d$;
		\item if $w_G(uv)=a-d$, then $u$ and $v$ are clones in $G$;
		\end{enumerate}
\end{definition}
\begin{remark}
Property (ii) above implies that the subgraph $G^{(a-d)}$ of $G$ consisting of edges of multiplicity $a-d$ is a disjoint union of cliques.
\end{remark}
\begin{proposition}\label{prop: reducing to good graphs with clique decompositions}
	Let $G\in\mathcal{G}(n, s, \Sigma_{r,d}(a,s))$. Then there exists $G'\in\mathcal{H}(n, s, \Sigma_{r,d}(a,s))$ such that $P(G)\leq P(G')$.
\end{proposition}
\begin{proof}
We modify $G$ in two phases, the first to raise the minimum edge multiplicity to $a-d$, and the second to ensure that all vertices joined by an edge of multiplicity $a-d$ are clones of each other. %We are grateful to a referee for a suggestion to correct an error in the initial argument for the second phase.

\textbf{First phase:} suppose $G$ contains an edge $uv$ of multiplicity $w_G(uv)<a-d$. We define a new multigraph $G_1$ from $G$ by increasing the multiplicity of $uv$ to $a-d$ and keeping all other edge multiplicities unchanged. Clearly $P(G_1)>P(G)$. We claim that in addition, like $G$, $G_1$ belongs to $\mathcal{G}(n,s, \Sigma_{r,d}(a,s))$.

Indeed, we clearly have $G_1 \in\mathcal{F}(n,2, a+1)$. Suppose $G_1\in \mathcal{F}(n, s', \Sigma_{r,d}(a,s'))$ for some $s'$: $2\leq s'<s$. Consider an $(s'-1)$-set $X\subseteq V\setminus\{u,v\}$. Then
\begin{align*}
S(G_1[X\cup\{u,v\}])&= S(G[X\cup\{u\}])+S(G[X, \{v\}])+ w_{G_1}(uv)\\
&\leq  \Sigma_{r,d}(a, s')+ (s'-1)(a+1)+ a-d=\Sigma_{r,d}(a,s')+s'(a+1)-(d+1),	
\end{align*}
which is less than $\Sigma_{r,d}(a, s'+1)$ by~\eqref{eq: bound on Sigma diff}. This immediately implies that $G_1\in \mathcal{F}(n, s'+1, \Sigma_{r,d}(a, s'+1))$ as well. Thus $G_1 \in \bigcap_{s'=2}^s \mathcal{F}(n, s', \Sigma_{r,d}(a,s'))= \mathcal{G}(n, s, \Sigma_{r,d}(a,s))$ as required.

Sequentially increasing the edge multiplicities of edges with $w_G(uv)<a-d$ in this way, we have that after at most $\binom{n}{2}$ steps we have produced a multigraph $G_2\in \mathcal{G}(n,s, \Sigma_{r,d}(a,s))$ with $P(G_2)\geq P(G)$ in which all edge multiplicities are at least $a-d$.

\textbf{Second phase:} we shall go through the multigraph $G_2$ in several passes. While there exist edges $uv$ in $G_2$ such that $w_{G_2}(uv)=a-d$ and $u,v$ are not clones of each other in $G_2$, we run the following algorithm:
\begin{enumerate}[(1)]
\item among all vertices of $G_2$ incident with such edges, we select one with maximum product-degree in $G_2$, and denote it by $u$;

\item we set $B_u$ to be the collection of vertices in $V(G_2)$ that are joined to $u$ by an edge of multiplicity $a-d$ and are not clones of $u$. While $B_u$ is non-empty, we pick a vertex $v\in B_u$ and modify $G_2$ by changing the multiplicity of $vw$ to $w_{G_2}(uw)$ for all $w\in V(G_2)\setminus\{u,v\}$ -- in other words, we replace $v$ by a clone of $u$. 
\end{enumerate}
Observe that each time we select $v\in B_u$ and replace it by a clone of $u$ in an iteration of Step (2) of our algorithm, the value of $p_{G_2}(v')$ does not increase for any $v'\in N^{(a-d)}(u)$. Indeed, the multiplicity of $vv'$ is changed to $a-d\leq w_{G_2}(vv')$. It follows in particular that after our change we still have $p_{G_2}(u)\geq p_{G_2}(v')$ for all $v'\in N^{(a-d)}(u)$. This ensures that our procedure does not decrease the value of $P(G_2)$ (since our modification of the graph changed this product by a multiplicative factor of $p_{G_2}(u)/p_{G_2}(v)\geq 1$).

Further, each time we replace some $v\in B_u$ by a clone of $u$ in an iteration of Step (2), it is easy to check that after our modification, the multigraph $G_2$ still lies in $\mathcal{G}(n,s, \Sigma_{r,d}(a,s))$. Indeed, this can be shown in exactly the same way that we proved $G_1\in \mathcal{G}(n,s, \Sigma_{r,d}(a,s))$ in the first phase.  Also if $v'\in N^{(a-d)}(u)\setminus B_u$, then by definition the multiplicities of $uv$ and $vv'$ were the same, so $v'$ remains a clone of $u$. In particular after at most $\vert B_u\vert<n$ iterations of this procedure, $B_u$ becomes empty. When this occurs, we have that all vertices in $N^{(a-d)}(u)$ are clones of $u$ (and of each other), and in particular $C_u:=\{u\}\cup N^{(a-d)}(u)$ forms an isolated clique in ${G_2}^{(a-d)}$ (i.e.\ all edges from $C_u$ to $V(G_2)\setminus C_u$ have multiplicity strictly greater than $a-d$ while all edges in $C_u$ have multiplicity equal to $a-d$)  and $w_{G_2}(uw)=w_{G_2}(u'w)$ for all $u,u'\in C_u$, $w\notin C_u$. Both of these properties are maintained in all subsequent iterations of Steps (1)--(2), from which it follows that no vertex of $C_u$ will ever again be selected in an iteration of Step (1). Thus our algorithm will terminate after at most $n/2$ iterations.

 The final multigraph $G'$ obtained when our algorithm terminates then has all the claimed properties: $P(G')\geq P(G_2)\geq P(G)$, $G'\in \mathcal{G}(n,s, \Sigma_{r,d}(a,s))$, and whenever $w_{G'}(uv)=a-d$, $u$ and $v$ are clones in $G'$.
\end{proof}
\noindent We now combine Propositions~\ref{prop: good multigraphs} and~\ref{prop: reducing to good graphs with clique decompositions} with Proposition~\ref{prop: degree removal} to show that to prove Theorem~\ref{theorem: d=1 asymptotic} it will be enough to restrict our attention to multigraphs from the `nice' family $\mathcal{H}(n,s, \Sigma_{r,d}(a,s))$ rather than the whole of $\mathcal{F}(n,s, \Sigma_{r,d}(a,s))$, and to show that these multigraphs contain low product-degree vertices.
\begin{proposition}\label{prop: degree removal final form}
	If for all $n\in \mathbb{N}$ and every $G \in \mathcal{H}(n+1,s,\Sigma_{r,d}(a,s))$ there exists 	$v\in V(G)$ with 
\[p_G(v)\leq a^{n} \left(\frac{a+1}{a}\right)^{\left(\frac{r-2+x_{\star}(r,d)}{r-1}\right)n +o(n)},\]
then
\[\textrm{ex}_{\Pi}(n,s,\Sigma_{r,d}(a,s))= \left(\Pi_{r,d}(a,n)\right)^{1+o(1)}.\]	
\end{proposition}
\begin{proof}
By Proposition~\ref{prop: degree removal}, it is enough to show that~\eqref{eq: degree removal inequality} holds for all $n$. Consider $G\in \mathcal{F}(n+1,s, \Sigma_{r,d}(a,s))$ with $P(G)=\textrm{ex}_{\Pi}(n+1,s,\Sigma_{r,d}(a,s))$.

Suppose $G\notin \mathcal{G}(n+1, s, \Sigma_{r,d}(a,s))$. Then by Proposition~\ref{prop: good multigraphs}, there is a vertex $v$ in $G$ with $p_G(v)\leq a^n\left(\frac{a+1}{a}\right)^{\frac{r-2}{r-1}n + O(1)}$. By removing $v$ from $G$ to obtain the multigraph $G-v\in \mathcal{F}(n,s, \Sigma_{r,d}(a,s))$, we see that
\begin{align*}
\textrm{ex}_{\Pi}(n+1,s,\Sigma_{r,d}(a,s)) =P(G)=P(G-v)p_{G}(v)\leq \mathrm{ex}_{\Pi}\left(n, s, \Sigma_{r,d}(a,s)\right) a^{n} \left(\frac{a+1}{a}\right)^{\frac{r-2}{r-1}n + O(1)},
\end{align*}
and~\eqref{eq: degree removal inequality} is satisfied.

On the other hand, suppose that $G\in \mathcal{G}(n+1, s, \Sigma_{r,d}(a,s))$. Then by Proposition~\ref{prop: reducing to good graphs with clique decompositions} there exists $G' \in 
\mathcal{H}(n+1, s, \Sigma_{r,d}(a,s))$ with $P(G)\leq P(G')$. Let $v$ be a vertex with minimum product-degree in $G'$. Removing $v$ from $G'$ to obtain the multigraph $G'-v\in \mathcal{H}(n,s, \Sigma_{r,d}(a,s))$ and using our assumption to bound $p_{G'}(v)$, we have
\begin{align*}
\textrm{ex}_{\Pi}(n+1,s,\Sigma_{r,d}(a,s)) \leq P(G')=P(G'-v)p_{G'}(v)\leq \mathrm{ex}_{\Pi}(n, s, \Sigma_{r,d}(a,s)) a^{n} \left(\frac{a+1}{a}\right)^{\left(\frac{r-2+x_{\star}(r,d)}{r-1}\right)n +o(n)}
\end{align*}
and see again that~\eqref{eq: degree removal inequality} is satisfied. The result follows.
\end{proof}

\section{Optimisation in neighbourhoods}\label{section: optimisation in neighbourhoods}
Throughout this section, let $r,a,d $ be positive integers with $r\geq 2$, $a>d\geq 1$. Set $s= r(d+1)+2$. Let $G\in \mathcal{H}(n,s, \Sigma_{r+1,d}(a,s))$. Let $x\in V(G)$. Recall that $N^{(m)}(x)$ is the collection of vertices sending an edge of multiplicity $m$ to $x$ in $G$. Set $X:=N^{(a-d)}(x)\cup\{x\}$,  $Y:=N^{(a+1)}(x)$ and $Z:=V(G)\setminus\left(X\cup Y\right)$. Let $\vert X\vert=\alpha n$, $\vert Y\vert=\beta n$. We begin by proving a general lemma which shows that if $G[Y]$ has a vertex with low product-degree, then so does $G$. 
\begin{lemma}\label{lemma: if a+1 neighbourhoods are sparse, done}
 If there exists $y\in Y$ such that
\[p_{Y}(y) \leq a^{\beta n}\left(\frac{a+1}{a}\right)^{ \left(\frac{r-2+x_{\star}(r, d)}{r-1}\right)\beta n+o(n)},\]
then $G$ contains a vertex $v$ with
\[p(v)\leq a^n \left(\frac{a+1}{a}\right)^{\left(\frac{r-1+x_{\star}(r+1,d)}{r}\right)n +o(n)}.\]
\end{lemma}
\begin{proof}
We have
\begin{align}
p(x)&\leq (a-d)^{\vert X\vert -1}(a+1)^{\vert Y\vert}a^{n-\vert X\vert - \vert Y\vert}
%\notag \\
%&
= a^n\left(\frac{a+1}{a}\right)^{\left(-\alpha \frac{\log(a/(a-d))}{\log((a+1)/a) }+ \beta   \right)n +O(1)} \label{eq: first optim lemma, bound on px}
\end{align}
and, by our assumption on $y$, 
\begin{align}
p(y)\leq (a+1)^{n-\vert Y\vert} p_Y(y)%\notag \\
%&
\leq a^n \left(\frac{a+1}{a}\right)^{\left(1-\left(\frac{1-x_{\star}(r, d)}{r-1}\right)\beta \right)n +o(n)}\label{eq: first optim lemma, bound on py}.
\end{align}
Now the maximum over all $\alpha, \beta\geq 0$  satisfying $\alpha+\beta\leq 1$ of the function
\[\min\Bigl\{-\alpha \frac{\log \left(\frac{a}{a-d}\right)}{\log\left(\frac{a+1}{a}\right) }+ \beta,  1-\left(\frac{1-x_{\star}(r, d)}{r-1}\right)\beta \Bigr\}\]
is attained at $\alpha=0$ and $\beta$ satisfying
\[\beta\left(1+ \frac{(1-x_{\star}(r, d))}{r-1} \right)=1.\]
Rearranging terms we see the maximum is precisely equal to
\[\frac{1}{1+ \frac{(1-x_{\star}(r, d))}{r-1}}= \frac{r-1}{r -x_{\star}(r,d)},\]
which by~\eqref{eq: recurrence for xstar} and rearranging terms again is equal to $(r-1+x_{\star}(r+1,d))/r$. Combining the result of this optimisation with the bounds~\eqref{eq: first optim lemma, bound on px} and~\eqref{eq: first optim lemma, bound on py} on $p(x)$ and $p(y)$, we get that
\[\min\{p(x),p(y) \}\leq   a^n \left(\frac{a+1}{a}\right)^{\left(\frac{r-1+x_{\star}(r+1,d)}{r}\right)n +o(n)}, \]
thereby proving the lemma.
\end{proof}

As indicated in the introduction, a key part of our proof strategy will be to consider certain `good' $r$-partite structures inside $G[Y]$, which we define below.
\begin{definition}\label{def: good subgraph}
Let $H=(V_H, E_H)$ be an ordinary graph. A \emph{good} copy of $H$ in $G$ is a set $X\subseteq V(G)$ of $\vert V_H\vert$ vertices such that (i) all edges in $X^{(2)}$ have multiplicity at least $a$ in $G$, and (ii) the edges in $X^{(2)}$ with multiplicity $a+1$ form a graph isomorphic to $H$. 
\end{definition}
We use `$G$ contains a good $H$' as a shorthand for `$G$ contains a good copy of $H$'. Good complete $r$-partite graphs will play a key role in our proof.  For integers $r, t_1, t_2, \ldots, t_r>0$, we let  $K_r(t_1,t_2, t_3, \ldots ,t_r)$ denote the complete $r$-partite (ordinary) graph with part-sizes $t_1, t_2, \ldots, t_r$. It will be convenient to have a slightly more compact  notation for such graphs. For $0\leq r'\leq r$, we write $K_r(\mathbf{t'}^{(r')}\mathbf{t}^{(r-r')})$ to denote the complete $r$-partite graph in which the first $r'$ parts have size $t'$ and the last $r-r'$ parts have size $t$. Similarly, we write $K_r(\mathbf{t})$ to denote an $r$-partite structure in which all $r$ parts have size $t$; and when $t=1$, we just write $K_r$ for the complete graph on $r$ vertices.
\begin{lemma}\label{lemma: if contain good Kr structure, then done}
Suppose $R\geq d+1$ is a positive integer such that 
\begin{align}\label{eq: R-condition}
(a-d+i)^R\leq (a+1)^{R-d+i-1}(a-d)^{d-i+1} && \textrm{for all }i\in[d].
\end{align}
Then if $G[Y]$ contains a good $K_r(\mathbf{R})$, $G$ must contain a vertex $v$ with
	\[p(v)\leq a^n \left(\frac{a+1}{a}\right)^{\left(\frac{r-1+x_{\star}(r+1,d)}{r}\right)n +o(n)}.\]
\end{lemma}
\begin{remark}\label{remark: bound on R}
Clearly~\eqref{eq: R-condition} is satisfied for all $R$ sufficiently large. Indeed, for $R= d(1+d+d^2)$, we have that for all $i\in [d]$,
\begin{align*}
(a+1)^{R-d+i-1}(a-d)^{d-i+1}&\geq (a+1)^{d(d+d^2)}(a-d)^{d}> \left(\left(a^{d+d^2} +(d+d^2)a^{d+d^2-1}\right)(a-d) \right)^d\\
&=\left(a^{1+d+d^2}+a^{d+d^2-1}\left(ad^2-d^2(d+1)\right) \right)^d\geq a^{d(1+d+d^2)}\geq (a-d+i)^{R},\end{align*}
where in the penultimate inequality we used $a\geq d+1$. In particular for $d=1$ we have that $R=3$ suffices --- a fact we will use in the proof of Theorem~\ref{theorem: d=1 asymptotic}.
\end{remark}
\begin{proof}
Let $W$ be a set of $rR$ vertices in $Y$ inducing a good $K_r(\mathbf{R})$, and let $\sqcup_{i=1}^rW_i$ be the associated partition of $W$ into  $R$-sets.  Since $G\in \mathcal{H}(n, s, \Sigma_{r+1,d}(a, s))$ and $R\geq d+1$, the following hold:
\begin{enumerate}[(i)]
	\item the graph $G^{(a+1)}$ is $K_{r+2}$-free. In particular if a vertex $v\in V\setminus \left(\{x\}\cup W\right)$ sends an edge of multiplicity $a+1$ to $x$, then it can send at most $(r-1)R$ edges of multiplicity $a+1$ into $W$.
	\item if $v\in V\setminus \left(\{x\}\cup W\right)$ sends an edge of multiplicity $a-d$ to $x$, then it sends exactly $rR$ edges of multiplicity $a+1$ into $W$.
	\item for $i\in [d]$, $(r(d-i+2)+2)$-sets in $G$ support at most $\Sigma_{r+1,d}(a, r(d-i+2)+2)$ edges. Now by~\eqref{eq: bound on differnence of sigmas d, d-i}, we have $\Sigma_{r+1,d}(a, r(d-i+2)+2)< \Sigma_{r+1,d-i}(a, r(d-i+2)+2)$. Further, we know by Proposition~\ref{prop: sum-extremal subgraphs, threshold for 2 vertices in extremal part} that there is a multigraph $H$ in $\mathcal{T}_{r+1,d-i}(a, r(d-i+2)+2)$ with $e(H)=\Sigma_{r+1,d-i}(a, r(d-i+2)+2)$ and whose canonical partition satisfies $\vert V_0\vert = 2$. In particular if $v\in V\setminus \left(\{x\}\cup W\right)$ sends an edge of multiplicity  $a-d+i$ to $x$, then at least one of the parts $W_i$ must receive at most $d-i+1$ edges of multiplicity $a+1$ from $v$. Indeed otherwise we could select $d-i+2$ vertices from each of the parts $W_i$ to form an $r(d-i+2)$ set $W'$ such that
	\[\Sigma_{r+1,d}(a, r(d-i+2)+2)\geq e(G[W'\cup\{x,v\}])\geq e(H)=\Sigma_{r+1,d-i}(a, r(d-i+2)+2),\] contradicting~\eqref{eq: bound on differnence of sigmas d, d-i}. Thus we have $S(v,W)\leq (a+1)\left((r-1)R+ d-i+1\right) + a\left(R-d+i-1\right)$, which by the integral AM-GM inequality (\ref{prop: integral AM-GM}) implies that $P(v, W) \leq a^{R-d+i-1}(a+1)^{(r-1)R+d-i+1}$.
\end{enumerate} 	
Now consider the quantity $p:= \left(p_G(x)\right)^{x_{\star}(r+1, d)} \left(\prod_{w\in W}p_G(w)\right)^{\frac{1-x_{\star}(r+1,d)}{rR}}$. By the observations (i)--(iii)	above, the contribution to $p$ made by a vertex $v\in V\setminus \left(\{x\}\cup W\right)$ is at most
\begin{align*}
\left\{\begin{array}{ll}
a \left( \frac{a+1}{a}\right)^{x_{\star}(r+1,d)}\left( \frac{a+1}{a}\right)^{\frac{r-1}{r}(1-x_{\star}(r+1,d))}= a \left( \frac{a+1}{a}\right)^{\frac{r-1+ x_{\star}(r+1,d)}{r}} & \textrm{if }w_G(xv)=a+1\\
a \left( \frac{a+1}{a}\right)^{-\frac{\log\left(\frac{a}{a-d}\right)}{\log\left(\frac{a+1}{a}\right)}x_{\star}(r+1,d)}\left( \frac{a+1}{a}\right)^{(1-x_{\star}(r+1,d))}= a \left( \frac{a+1}{a}\right)^{\frac{r-1+ x_{\star}(r+1,d)}{r}}& \textrm{if }w_G(xv)=a-d\\
a\left(\frac{a+1}{a}\right)^{-\frac{\log\left(\frac{a}{a-d+i}\right)}{\log\left(\frac{a+1}{a}\right)}x_{\star}(r+1,d)}
\left( \frac{a+1}{a}\right)^{\left(1-\frac{R-d+i-1}{rR}\right)(1-x_{\star}(r+1,d))} & \textrm{if  }w_G(xv)=a-d+i, \ i \in [d]
\end{array}
 \right. 
\end{align*}
Now our assumption~\eqref{eq: R-condition} on $R$ ensures that
\begin{align*}
-&\frac{\log\left(\frac{a}{a-d+i}\right)}{\log\left(\frac{a+1}{a}\right)}x_{\star}(r+1,d) + \left(1-\frac{R-d+i-1}{rR}\right)(1-x_{\star}(r+1,d))\\
&\quad =\frac{r-1+ x_{\star}(r+1,d)}{r} + \frac{\log\left( \frac{(a-d+i)^R}{(a+1)^{R-d+i-1} (a-d)^{d-i+1}}\right)}{R \log\left(\frac{(a+1)^{r+1}}{(a-d)^{r}a}\right)}\leq \frac{r-1+ x_{\star}(r+1,d)}{r}.
\end{align*}
Thus in all three cases, $v$ contributes at most $a \left( \frac{a+1}{a}\right)^{\frac{r-1+ x_{\star}(r+1,d)}{r}}$ to $p$. Since $p$ is a weighted geometric mean of the product-degrees of the vertices in $\{x\}\cup W$, it follows that there is some vertex $u\in \{x\}\cup W$ satisfying
\begin{align*}
p(u)\leq p\leq a^n \left(\frac{a+1}{a}\right)^{\left(\frac{r-1+x_{\star}(r+1,d)}{r}\right)n +o(n)},
\end{align*}
as required.	
\end{proof}	
We next prove an optimisation lemma that will be a key tool in Section~\ref{section: finding r-partite} when we try to find good $r$-partite structures in $Y$ with sufficiently large part-sizes. 
\begin{lemma}\label{lemma: if exists lowdeg vertex in YcupZ then done}
If $(a+1)^{r}(a-d)\geq a^{r+1}$ and there exist $y\in Y\cup Z$ such that
\[p_{Y\cup Z}(y) \leq a^{(1-\alpha)n} \left(\frac{a+1}{a}\right)^{\frac{r-1}{r}\beta n+ \frac{r}{r+1}(1-\alpha-\beta)n +o(n) },\]
then $G$ contains a vertex $v$ with
\[p(v)\leq a^n \left(\frac{a+1}{a}\right)^{\left(\frac{r-1+x_{\star}(r+1,d)}{r}\right)n +o(n)}.\]
\end{lemma}
\begin{remark}
	The condition $(a+1)^{r}(a-d)\geq a^{r+1}$ will always be satisfied when $d=1$ and $r\geq 2$ (which is all that we need in a proof of Theorem~\ref{theorem: d=1 asymptotic}, since the base case $(r,d)=(2,1)$ was proved in~\cite{DayFalgasRavryTreglown20+}). For larger $d$, however, this condition will only be satisfied for sufficiently large $r$ --- it is e.g. easy to check $r\geq d(d+1)$ will do with a calculation similar to that in Remark~\ref{remark: bound on R}. This suggests a more precise form of Lemma~\ref{lemma: if exists lowdeg vertex in YcupZ then done} may be one of the tools necessary to tackle the cases $d\geq 2$ of Conjecture~\ref{conjecture: main conjecture}. 
\end{remark}
%\begin{remark}
%	We already know that $G[Y]$ cannot contain a good $K_{r+1}$ (since otherwise $G$ contains a good $K_{r+2}$), which immediately implies by Tur\'an's theorem that there is some vertex $y_1\in Y$ whose product-degree in $G[Y]$ is at most $a^{\beta n}\left(\frac{a+1}{a}\right)^{\frac{r-1}{r}\beta n+o(n)}$. Similarly, $G[Z]$ cannot contain a good $K_{r+2}$, whence there is some vertex $y_2\in Z$ whose product-degree in $G[Z]$ is at most $a^{(1-\alpha-\beta) n}\left(\frac{a+1}{a}\right)^{\frac{r}{r+1}(1-\alpha-\beta) n+o(n)}$. The assumption in Lemma~\ref{lemma: if exists lowdeg vertex in YcupZ then done} is thus simply that we can find some vertex $y$ whose product-degrees in $Y$ and $Z$ simultaneously satisfy the upper bounds for $y_1$ and $y_2$.
%\end{remark}
\begin{proof}
We have
\begin{align}
p(x)&\leq (a-d)^{\vert X\vert -1}(a+1)^{\vert Y\vert}a^{n-\vert X\vert - \vert Y\vert}
%\notag \\
%&
= a^n\left(\frac{a+1}{a}\right)^{\left(-\alpha \frac{\log(a/(a-d))}{\log((a+1)/a) }+ \beta   \right)n +O(1)} \label{eq: second optim lemma, bound on px}
\end{align}
and, by our assumption on $y$,
\begin{align}
p(y)\leq (a+1)^{\vert X\vert } p_{Y\cup Z}(y)%\notag \\
%&
\leq a^n \left(\frac{a+1}{a}\right)^{\alpha n+\frac{r-1}{r}\beta n+ \frac{r}{r+1}(1-\alpha-\beta)n +o(n) }\label{eq: second optim lemma, bound on py}.
\end{align}	
As in Lemma~\ref{lemma: if a+1 neighbourhoods are sparse, done} we perform some optimisation on the exponents of $(a+1)/a$ in~\eqref{eq: second optim lemma, bound on px} and~\eqref{eq: second optim lemma, bound on py} to bound $\min\{p(x), p(y)\}$. Set
\begin{align*}
f_1(\alpha, \beta)=-\alpha \frac{\log(a/(a-d))}{\log((a+1)/a) }+ \beta && \textrm{ and }&&f_2(\alpha, \beta)=\frac{r}{r+1}+\frac{\alpha}{r+1}- \frac{\beta}{r(r+1)}.\end{align*}
Let $f_3(\alpha, \beta)=\min\{f_1(\alpha, \beta), f_2(\alpha, \beta)\}$ and $S:=\{(\alpha, \beta)\in [0,1]^2:\ \alpha +\beta \leq 1 \}$.
\begin{claim}\label{claim: second optim lemma, max on the boundary}
The maximum of $f_3(\alpha, \beta)$ over $S$ is attained on the boundary $\{(\alpha, \beta): \ \alpha +\beta=1\}$.	
\end{claim}
\begin{proof}
Indeed, suppose $(\alpha, \beta) \in S$ is such that $\alpha+\beta <1$ . Then there exists some $\varepsilon>0$ such that the pair $(\alpha', \beta')$ given by $\alpha'=\alpha +\varepsilon \log\left(\frac{a+1}{a}\right)$  and  $\beta'=\beta+ \varepsilon \log\left(\frac{a}{a-d}\right)$
 lies in $S$. Now $f_1(\alpha', \beta')=f_1(\alpha, \beta)$ and \begin{align*}
f_2(\alpha', \beta')-f_2(\alpha, \beta)&=\frac{\varepsilon}{(r(r+1))}\log \left(\frac{(a+1)^{r}(a-d)}{a^{r+1}}\right)\geq 0,
\end{align*}
where the inequality follows from our assumption that $(a+1)^r(a-d)\geq a^{r+1}$. Thus $f_3(\alpha', \beta')\geq f_3(\alpha, \beta)$ and $\alpha'+\beta'>\alpha +\beta$. It immediately follows that the maximum of $f_3(\alpha, \beta)$ in $S$ is attained on the boundary $\alpha+\beta=1$, as claimed.
\end{proof}
By Claim~\ref{claim: second optim lemma, max on the boundary}, the maximum of $f_3(\alpha, \beta)$ over $S$ is the same as the maximum of $f_3(\alpha, 1-\alpha)$ over $\alpha\in [0,1]$. This is readily computed: \begin{align*}
f_1(\alpha, 1- \alpha)= 1-\alpha\frac{\log\left(\frac{a+1}{a-d}\right)}{\log\left(\frac{a+1}{a}\right)}, && f_2(\alpha, 1-\alpha)= \frac{r-1}{r}+\frac{\alpha}{r},
\end{align*}
and these two functions are respectively strictly decreasing and strictly increasing in $\alpha$, so that the maximum of $f_3(\alpha, 1-\alpha)$ is attained at $\alpha=x_{\star}(r+1,d)$, when the two functions are equal to $\frac{r-1+x_{\star}(r+1,d)}{r}$. Combining this optimisation result with the bounds on $p(x)$ and $p(y)$ given in~\eqref{eq: second optim lemma, bound on px} and~\eqref{eq: second optim lemma, bound on py}, we get that
\[\min\{p(x),p(y) \}\leq   a^n \left(\frac{a+1}{a}\right)^{\left(\frac{r-1+x_{\star}(r+1,d)}{r}\right)n +o(n)}, \]
thereby proving the lemma.
\end{proof}

\section{Proof of Theorem~\ref{theorem: d=1 asymptotic}}\label{section: finding r-partite}
We shall proceed by induction on $r$. The base case $r=2$ was proved in~\cite[Theorem 3.5]{DayFalgasRavryTreglown20+}. Suppose that we had proved Theorem~\ref{theorem: d=1 asymptotic} holds for all $(r',a)$ with $r'\leq r$ and $a\geq 2$, for some $r\geq 2$.

\begin{proof}[Proof of the inductive step]
Fix $a\in \mathbb{Z}_{\geq 2}$. Let $G\in \mathcal{H}\left(n,2r+2, \Sigma_{r+1,1}(a,2r+2)\right)$. Recall that by the definition of the family $\mathcal{H}$ this means all edges in $G$ have multiplicity $a-1$, $a$ or $a+1$, that $G^{(a-1)}$ is a disjoint union of cliques and that two vertices joined by an edge of multiplicity $a-1$ are clones of each other in $G$. These are key properties we shall use repeatedly in our proof.

Let $x\in V(G)$. Set $X:=N^{(a-1)}(x)\cup\{x\}$,  $Y:=N^{(a+1)}(x)$ and $Z:=V(G)\setminus\left(X\cup Y\right)= N^{(a)}(x)$. Let $\vert X\vert=\alpha n$, $\vert Y\vert=\beta n$.
\begin{definition}
	A vertex $v$ in $G$ is said to be \emph{product-poor} if $p(v)\leq a^n \left(\frac{a+1}{a}\right)^{\left(\frac{r-1+x_{\star}(r+1,1)}{r}\right)n +o(n)}$. Further, $v$ is said to be \emph{strictly product-poor} if there exists some constant $\delta>0$ such that $p(v)\leq a^n \left(\frac{a+1}{a}\right)^{\left(\frac{r-1+x_{\star}(r+1,1)}{r}-\delta\right)n  +o(n)}$.
\end{definition}
\noindent Our goal is to show that $G$ contains a product-poor vertex, which by Proposition~\ref{prop: degree removal final form} is enough to prove the inductive step. By considering the product-degree of $x$, we may thus assume that
\begin{align}\label{eq: beta bound}
\beta \geq \frac{r-1+x_{\star}(r+1,1)}{r} +o(1)> \frac{r-1}{r}.
\end{align}
\noindent Indeed, if this was not the case, then by~\eqref{eq: first optim lemma, bound on px} it would follow that $x$ is a product-poor vertex, and we would be done.

Before embarking on the main body of the proof, let us record the following elementary observations about neighbourhoods of vertices in $Y$.
	\begin{proposition}\label{prop: basic properties of Y nhoods} The following hold:
	\begin{enumerate}[(i)]
		\item $G^{(a+1)}$ is $K_{r+2}$-free;
		\item $G^{(a+1)}[Y]$ is $K_{r+1}$-free;
		\item for $R\geq 1, t\geq 0$, every $(rR+t)$-set in $Y$ spanning $\Sigma_{r,0}(a,rR+t)$ edges in $G[Y]$ induces a good $K_r(\mathbf{R}^{(r-t)}\mathbf{(R+1)}^{(t)})$ in $G[Y]$;  
		\item all edges from $Z$ to $Y$ have multiplicity at least $a$;
		\item 	if $z\in Z$ and $W$ is a $2r$-set in $Y$ inducing a good $K_r(\mathbf{2})$ in $G[Y]$, then $z$ sends at most $2r-1$ edges of multiplicity $a+1$ into in $W$.
	\end{enumerate}
\end{proposition}
\begin{proof}
Since $G\in \mathcal{H}(n,2r+2, \Sigma_{r+1,1}(a,2r+2))$, the graph $G^{(a+1)}$ must be $K_{r+2}$-free (else we have an $(r+2)$-set with strictly more than $\Sigma_{r+1,1}(a, r+2)= (a+1)\binom{r+2}{2}-1$ edges, a contradiction). This establishes part (i). In turn, part (i) implies that $Y$, being the $(a+1)$-neighbourhood of $x$, does not contain a $K_{r+1}$ in which all edges have multiplicity $a+1$, establishing part (ii).

Part (iii) follows from part (ii) and Tur\'an's theorem. Indeed, if $Y$ spans $\Sigma_{r,0}(a,rR+t)$ edges of $G$, then it needs to support at least $\mathrm{ex}(rR+t, K_{r+1})$ edges of multiplicity $a+1$ (since $G$ contains no edges of multiplicity greater than $a+1$). By Tur\'an's theorem, this implies that either $Y$ contains a copy of $K_{r+1}$, which contradicts (ii), or that $G^{(a+1)}[Y]$ is a copy of the $r$-partite Tur\'an graph on $rR+t$ vertices, with all other edges in $G[Y]$ having multiplicity exactly $a$ --- in other words, $G[Y]$ is a good $K_r(\mathbf{R}^{(r-t)}\mathbf{(R+1)}^{(t)})$, as claimed.

	Part (iv) follows from the fact that the multiplicity of edges from $Z$ to $x$ and from $Y$ to $x$ have different multiplicities. Thus vertices in $Z$ and $Y$ cannot be clones of each other, whence by definition of $\mathcal{H}(n, 2r+2, \Sigma_{r+1,1}(a, 2r+2) )$ they cannot be joined by edges of multiplicity $a-1$.

	For part (v), observe that otherwise $z,x$ together with $W$	 induces a good $K_{r+1}(\mathbf{2})$ in $G$, i.e. a $2r+2$-set spanning $\Sigma_{r+1,0}(a, 2r+2)=\Sigma_{r+1,1}(a, 2r+2)+1$ edges, a contradiction.
\end{proof}
Suppose $G[Y]$ does not contain a good $K_r(\mathbf{2})$. Then by Proposition~\ref{prop: basic properties of Y nhoods}~(iii), this implies that every $2r$-set in $Y$ spans at most $\Sigma_{r,0}(a, 2r)-1=\Sigma_{r,1}(a, 2r)$ edges. Then by our inductive hypothesis and averaging, $G[Y]$ must contain some vertex $y$ with product-degree
\[p_Y(y)\leq a^{\beta n}\left(\frac{a+1}{a}\right)^{\frac{r-2+x_{\star}(r, 1)}{r-1}\beta n+o(n)},\]
whence $G$ contains a product-poor vertex by Lemma~\ref{lemma: if a+1 neighbourhoods are sparse, done} and we are done.

We may thus assume that $G[Y]$ contains a good $K_r(\mathbf{2})$. By Lemma~\ref{lemma: if contain good Kr structure, then done} and Remark~\ref{remark: bound on R}, we may further assume that $G[Y]$ does not contain a good $K_r(\mathbf{3})$ (since otherwise $G$ contains a product-poor vertex). Let us then define $t$: $0\leq t\leq r-1$ to be the largest integer such that $G[Y]$ contains a good $K_r(\mathbf{2}^{(r-t)}\mathbf{3}^{(t)})$. Let $U_i=\{u_{i,1}, u_{i,2}\}$, $i\in [r-t]$, and $W_i=\{w_{i,1}, w_{i,2}, w_{i, 3} \}$ , $i\in [t]$ be $r$ disjoint sets of vertices in $Y$ that induce such a structure, with all edges in the $U_i^{(2)}$ and the $W_i^{(2)}$ having multiplicity $a$, and all other edges between these sets having multiplicity $a+1$. Set $U:=\bigcup_{i=1}^{r-t}U_i$ and $W:=\bigcup_{i=1}^{t}W_i$.

\begin{lemma}\label{lemma: t>0}
	Either $t>0$ or $G$ contains a product-poor vertex.
\end{lemma}
\begin{proof}
Suppose $t=0$. Fix a vertex $y\in Y\setminus U$. If $y$ sends at least one edge of multiplicity $a+1$ to each of the $r$ parts $U_1,\  U_2, \ldots , \ U_r$, then $G^{(a+1)}[U\cup \{y\}]$ contains a copy of $K_{r+1}$, contradicting Proposition~\ref{prop: basic properties of Y nhoods} (ii). Thus every vertex $y \in Y\setminus U$ sends edges of multiplicity $a+1$ to at most $r-1$ of the parts $U_i$, $i\in [r]$, and in particular sends at most $2(r-1)$ such edges into $U$ in total. Further, if $y$ sends exactly $2(r-1)$ edges of multiplicity $a+1$ into $U$, then there exists a unique part $U_i$, $i\in [r]$, to which it sends no such edge. By the maximality of $t$, at least one of the edges $y$ sends into this unique part $U_i$ must then have multiplicity $a-1$ (for otherwise $G$  would have a good  $K_r(\mathbf{2}^{(r-1)}\mathbf{3}^{(1)})$ living inside the set $U\cup\{y\}\subseteq Y$).

Since $(a-1)(a+1)<a^2$, it follows from the observations in the paragraph above that for all $y\in Y\setminus U$,
\[p_U(y)\leq \max\left\{(a+1)^{2(r-1)}a(a-1),\ (a+1)^{2(r-1)-1}a^3\right\}= (a+1)^{2(r-1)-1}a^3.\]
 Further, by Proposition~\ref{prop: basic properties of Y nhoods} (v), each vertex $z\in Z$ can send at most $2(r-1)+1$ edges of multiplicity $a+1$ into $U$, so that 
 \begin{align*}
 p_U(z)\leq (a+1)^{2r-1}a.
 \end{align*}
 By averaging over vertices in $U$, it follows that some $u\in U$ satisfies
\begin{align*}p_{(Y\setminus U)\cup Z}(u)&\leq  \left(\prod_{y\in Y\setminus U} p_U(y)\prod_{z\in Z} p_U(z)\right)^{\frac{1}{\vert U\vert}}\leq \left((a+1)^{2(r-1)-1}a^3\right)^{\frac{\vert Y\setminus U\vert}{2r}} \left((a+1)^{2r-1}a\right)^{\frac{\vert Z\vert}{2r} }\\
&= a^{(1-\alpha)n}\left(\frac{a+1}{a}\right)^{\left(\frac{r-1}{r}-\frac{1}{2r}\right)\beta n + \left(\frac{r-1}{r}+\frac{1}{2r}\right)(1-\alpha-\beta)n+O(1)}.\end{align*}
Appealing to the bound $\beta>(r-1)/r$ from~\eqref{eq: beta bound}, we have
\begin{align*}
\left(\frac{r-1}{r}\beta +\frac{r}{r+1}(1-\alpha-\beta) \right)&- \left(\left(\frac{r-1}{r}-\frac{1}{2r}\right)\beta  + \left(\frac{r-1}{r}+\frac{1}{2r}\right)(1-\alpha-\beta)\right)\\
& \geq \frac{\beta}{2r}- (1-\beta)\left(\frac{1}{2r}-\frac{1}{r(r+1)}\right)> \frac{1}{r^2}\left(\frac{r-2}{2}+\frac{1}{r(r+1)} \right)> 0.
\end{align*}
This implies that
\[p_{Y\cup Z}(u)= \left((a+1)^{2(r-1)}a\right) p_{(Y\setminus U)\cup Z}(u) \leq a^{(1-\alpha)n}\left(\frac{a+1}{a}\right)^{\frac{r-1}{r}\beta n +\frac{r}{r+1}(1-\alpha-\beta)n+O(1)},\]
whence $G$ contains a product-poor vertex by Lemma~\ref{lemma: if exists lowdeg vertex in YcupZ then done}.
\end{proof}
%By Lemma~\ref{lemma: t>0}, we may thus assume that $t>0$ and $W\neq \emptyset$.  
\begin{lemma}\label{lemma: t geq r-1}
	If $0<t\leq r-2$, then $G$ contains a product-poor vertex.
\end{lemma}
\begin{proof}
Suppose $0<t\leq r-2$ (and in particular $W\neq \emptyset$). Let $p_U$ and $p_W$ be the geometric-mean of the product-degrees $p_{Y\cup Z}(v)$ over $v\in U$ and $v\in W$ respectively,
$p_U:=\prod_{v\in U}p_{Y\cup Z}(v)^{\frac{1}{2(r-t)}}$ and $p_W:=\prod_{v\in W}p_{Y\cup Z}(v)^{\frac{1}{3t}}$.

Let us consider what contribution a vertex $y\in Y\setminus \left(U\cup W\right)$ can make to $p_U$ and $p_W$.  Recall that by Proposition~\ref{prop: basic properties of Y nhoods}(ii), $G^{(a+1)}[Y]$ is $K_{r+1}$ free, whence every vertex $y\in  Y$ can send an edge of multiplicity $a+1$ to at most $r-1$ of the parts $U_1, \ldots, U_{r-t}$, $W_1, \ldots W_t$. It follows that all such vertices $y$ fall into exactly one of the following three types.
%By the maximality of $t$ and the fact that $G^{(a+1)}[Y]$ is $K_{r+1}$-free, we have that one of the following must hold:
\begin{itemize}
	\item \textbf{Type Y1:} $y$ sends edges of multiplicity $a+1$ to all of $W$. This implies that there is at least one part $U_i$, $i\in [r-t]$, such that $y$ sends no edge of multiplicity $a+1$ into $U_i$, and in particular that $y$ sends at most $2(r-t-1)$ edges into $U$. Further, if $y$ sends exactly $2(r-t-1)$ edges into $U$ then the aforementioned part $U_i$ is unique, and it follows from the maximality of $t$ that $y$ sends an edge of multiplicity $a-1$ into $U_i$ (since otherwise $G[\{y\}\cup U\cup W]$ would contain a good $K_{r}(\mathbf{2}^{r-t-1}\mathbf{3}^{(t+1)})$, contradicting the maximality of $t$). Thus $p_U(y)$ is at most $(a+1)^{2(r-t-1)}a(a-1)$ if $y$ send $2(r-t-1)$ edges of multiplicity $a+1$ into $U$, and at most $(a+1)^{2(r-t-1)-1}a^3$ otherwise. Since $(a+1)(a-1)<a^2$, it  follows that $y$'s contribution to $p_U$ is at most $a\left(\frac{a+1}{a}\right)^{1-\frac{3}{2(r-t)}}$, while its contribution to $p_W$ is $a+1$;
	
	\item \textbf{Type Y2:} $y$ sends edges of multiplicity $a+1$ to at least $2(r-t-1)+1$ vertices in $U$. Then by the pigeon-hole principle, $y$ sends an edge of multiplicity $a+1$ to each of the parts $U_i$, $i\in [r-t]$, whence there is at least one part $W_j$, $j\in [t]$ that receives no edge of multiplicity $a+1$ from $y$. In particular, $y$ sends at most $3(t-1)$ edges of multiplicity $a+1$ to $W$. It follows that $y$ contributions to $p_U$ and $p_W$ are at most $(a+1)$ and $a\left(\frac{a+1}{a}\right)^{1-\frac{1}{t}}$ respectively;

	\item \textbf{Type Y3}: $y$ sends at most $2(r-t-1)$ edges of multiplicity $a+1$ into $U$ and at most $3t-1$ edges of multiplicity $a+1$ in $W$. Its contribution to $p_U$ and $p_W$ are thus at most $a\left(\frac{a+1}{a}\right)^{1-\frac{1}{r-t}}$ and $a\left(\frac{a+1}{a}\right)^{1-\frac{1}{3t}}$ respectively.
\end{itemize}
	We now turn our attention to the contributions of vertices $z\in Z$ to $p_U$ and $p_W$. Recall that by Proposition~\ref{prop: basic properties of Y nhoods}(iv), all edges from $Z$ to $U\cup W\subseteq Y$ have multiplicity $a$ or $a+1$. Further, by Proposition~\ref{prop: basic properties of Y nhoods}(v), each $z\in Z$ can send at most $2r-1$ edges of multiplicity $a+1$ into a $2r$-set inducing a copy of a good $K_r(\mathbf{2})$. In particular, if $z$ sends edges of multiplicity $a+1$ to all $2(r-t)$ vertices of $U$, it must be the case that there is one part $W_i$, $i\in [t]$ receiving at most one edge of multiplicity $a+1$ from $z$, and thus $z$ can send at most a total of $3(t-1)+1$ edges of multiplicity $a+1$ into $W$. It follows from these observations that vertices $z\in Z$ fall into one of the following two mutually exclusive types:
	\begin{itemize}
		\item \textbf{Type Z1:} $z$ sends at most $2(r-t-1)+1$ edges of multiplicity $a+1$ into $U$, whence its contributions to $p_U$ and $p_W$ are at most $a\left(\frac{a+1}{a}\right)^{1-\frac{1}{2(r-t)}}$ and $a+1$ respectively;
		\item \textbf{Type Z2}: $z$ sends $2(r-t)$ edges of multiplicity $a+1$ into $U$ and at most $3(t-1)+1$ such edges into $W$, whence its contributions to $p_U$ and $p_W$ are at most $a+1$ and $a\left(\frac{a+1}{a}\right)^{1-\frac{2}{3t}}$ respectively. 
	\end{itemize}
For $i\in [3]$ let $\theta_i$ be the proportion of vertices in $Y\setminus\left(U\cup W\right)$ of Type Yi, and let $\phi$ be the proportion of vertices in $Z$ of Type Z1. Plugging in our upper bounds on the contributions of the vertices of the various types to $p_U$ and $p_W$, and recalling that $\vert Y\vert =\beta n$, $\vert Z\vert =(1-\alpha-\beta) n$ and $\theta_1+\theta_2+\theta_3=1$, we have 
\begin{align*}
\frac{p_U}{a^{\vert Y\cup Z\vert}}&\leq   		\left(\frac{a+1}{a}\right)^{ \left(\theta_1\left(1-\frac{3}{2(r-t)}\right) +\theta_2 + \theta_3 \left(1-\frac{1}{r-t}\right)\right)\vert Y\vert + \left(\phi \left(1-\frac{1}{2(r-t)}\right)+(1-\phi)\right)\vert Z\vert + \vert U\cup W\vert}\\
&=\left(\frac{a+1}{a}\right)^{\left(f_U\beta  + g_U(1-\alpha-\beta)\right)n +O(1)},
\end{align*}
and
\begin{align*}
\frac{p_W}{a^{\vert Y\cup Z\vert}}&\leq   		\left(\frac{a+1}{a}\right)^{ \left(\theta_1+\theta_2 \left(1-\frac{1}{t}\right)+ \theta_3 \left(1-\frac{1}{3t}\right)\right)\vert Y\vert + \left(\phi +(1-\phi)\left(1-\frac{2}{3t}\right)\right)\vert Z\vert+\vert U\cup W\vert}\\ &=\left(\frac{a+1}{a}\right)^{\left(f_W\beta  + g_W(1-\alpha-\beta)\right)n +O(1)},
\end{align*}
where the functions $f_U=f_U(\theta_1, \theta_2)$, $f_W=f_W(\theta_1, \theta_2)$, $g_U=g_U(\phi)$, $g_W=g_W(\phi)$ are given by 
\begin{align*}
f_U:=
&1-\frac{1}{r-t} -\frac{\theta_1}{2(r-t)} +\frac{\theta_2}{r-t}, &&
f_W:=1-\frac{1}{3t}+\frac{\theta_1}{3t}-\frac{2\theta_2}{3t} \\
g_U:=&1-\frac{\phi}{2(r-t)},  && g_W:=1-\frac{2}{3t}+\frac{2\phi}{3t}.
\end{align*}
We shall consider a weighted geometric mean of $p_U$ and $p_W$ to deduce from the information above that $G$ contains a product-poor vertex.  In order to do so, we shall need a constraint on the values of $\theta_1, \theta_2$, which will follow from the claim below.
\begin{claim}\label{claim: thetai constraint}
If $f_W \leq \frac{r-2}{r-1}$, then $G$ contains product-poor vertex.
\end{claim}
\begin{proof}
Our classification of vertices of $Y\setminus (U\cup W)$ into types Y1, Y2 and Y3 and our bound on their contributions to $p_W$ imply that 
\begin{align*}
\left(\prod_{w\in W}\frac{p_Y(w)}{a^{\vert Y\vert}}\right)^{\frac{1}{\vert W\vert}} \leq  \left(\frac{a+1}{a}\right)^{f_W\vert Y\vert +\vert U\cup W\vert}.
\end{align*}	
Thus if $f_W\leq \frac{r-2}{r-1}$, then substituting in $\vert Y\vert =\beta n$ in the bound above, we see by averaging that there exists $w\in W$ with  product-degree at most $a^{\frac{1}{r-1}\beta n}(a+1)^{\frac{r-2}{r-1}\beta n +O(1)}$ in $Y$. By Lemma~\ref{lemma: if a+1 neighbourhoods are sparse, done} this $w$'s existence implies $G$ contains a product-poor vertex, proving our claim.
\end{proof}
\noindent We may thus assume that $f_W>\frac{r-2}{r-1}$, which by rearranging terms implies that
\begin{align}\label{eq: bound on theta1-Rtheta2}
\theta_1-2\theta_2> 1-\frac{3t}{r-1}.
\end{align}
Now since $\beta>\frac{r-1}{r}$ by~\eqref{eq: beta bound}, we have:
\begin{align*}
\left(\frac{r-1}{r}\beta +\frac{r}{r+1}(1-\alpha-\beta) \right)&-\left(\left(\frac{r-t}{r}f_U+\frac{t}{r}f_W\right)\beta + \left(\frac{r-t}{r}g_U+\frac{t}{r}g_W\right)(1-\alpha-\beta)\right)\\
&=\beta \frac{2+\theta_1-2\theta_2}{6r}-(1-\alpha-\beta)\left(\frac{1}{3r}+\frac{\phi}{6r}-\frac{1}{r(r+1)}\right)\\
&> \frac{1}{6r^2}\left((r-1)\left(2+\theta_1-2\theta_2\right)- \left(2+\phi -\frac{6}{r+1)}\right) \right)\\
&> \frac{1}{6r^2}\left(3(r-1-t)   - 3\right)\geq 0,
\end{align*}
where the last three inequalities follow from~\eqref{eq: bound on theta1-Rtheta2}, $\phi\leq 1$ and the assumption $t\leq r-2$. Consider now $p:=\left(p_U\right)^{\frac{r-t}{r}}\left(p_W\right)^{\frac{t}{r}}$. By the inequality we have just proved, and our bounds on $p_U$, $p_W$ in terms of $f_U$, $f_W$, $g_U$ and $g_W$,
\begin{align*}
p= a^{\vert Y\cup Z\vert} \left(\frac{p_U}{a^{\vert Y\cup Z\vert}}\right)^{\frac{r-t}{r}}\left(\frac{p_W}{a^{\vert Y\cup Z\vert}}\right)^{\frac{t}{r}}&\leq a^{(1-\alpha)n}\left(\frac{a+1}{a}\right)^{\frac{r-t}{r}\left(f_U \beta +g_U(1-\alpha-\beta)\right)n + \frac{t}{r}\left(f_W\beta + g_W(1-\alpha-\beta)\right)n +O(1)}\\ 
&\leq a^{(1-\alpha)n}\left(\frac{a+1}{a}\right)^{\frac{r-1}{r}\beta n+ \frac{r}{r+1}(1-\alpha-\beta)n+O(1)}.\end{align*}
Since $p$ is a weighted geometric mean of the product-degrees of vertices from $U\cup W$ in $Y\cup Z$, it follows some $v\in U\cup W$ satisfies $p_{Y\cup Z}(v)\leq p$. Given our upper-bound on $p$, this implies by Lemma~\ref{lemma: if exists lowdeg vertex in YcupZ then done} that $G$ contains a product-poor vertex, and we are done. 
\end{proof}
By Lemma~\ref{lemma: t geq r-1}, we may thus assume $t=r-1$. This is by far the most delicate case. As a first step, we show that it is enough for us to find an `almost' good $K_r(\mathbf{3})$. 
\begin{definition}
Let $H$ be the (ordinary) graph obtained from $K_r(\mathbf{3})$ by deleting one edge. We say that $G$ contains an \emph{almost good} $K_r(\mathbf{3})$ if it contains a good copy of $H$.
\end{definition}
\begin{lemma}\label{lemma: almost good is good enough}
If $G[Y]$ contains an almost good $K_r(\mathbf{3})$, then $G$ conducts a product-poor vertex.	
\end{lemma}
\begin{proof} This is somewhat similar to the proof of Lemma~\ref{lemma: t geq r-1}. Let $U_i=\{u_{i,1}, \ldots, u_{i, 3}\}$, $i\in [2]$, and $W_i=\{w_{i,1}, \ldots, w_{i, 3}\}$, $i\in [r-2]$ be $r$ disjoint sets in $Y$ such that all edges from $U_i$ to $W_j$ and all edges from $U_1$ to $U_2$ except $u_{1, 3}u_{2, 3}$ have multiplicity $a+1$, and all other edges inside $U=U_1\cup U_2$ and $W=\bigcup_{i=1}^{r-2}W_i$ have multiplicity $a$ (so $U\cup W$ induces an almost good $K_r(\mathbf{3})$ in $G[Y]$).

\textbf{Case 1: $\mathbf{r=2}$}.  Let $p_U$ be the geometric mean of the $p_{Y\cup Z}(u)$, $u\in U$. Since $G^{(a+1)}[Y]$ is $K_3$-free, for every vertex $y\in Y$ we must have that $N^{(a+1)}(Y)\cap U$ is a subset of one of $U_1$ or $U_2$ or $\{u_{1,3},u_{2,3}\}$. In particular $y$ can send at most three edges of multiplicity $a+1$ into $U$, and can only do so if $N^{(a+1)}(y)\cap U= U_i$ for some $i\in [2]$. Further, if $N^{(a+1)}(y)\cap U= U_i$ for some $i\in [2]$ and $y\notin U_{3-i}$, then $y$ must send an edge of multiplicity $a-1$ into $U_{3-i}$, for otherwise $\{y, u_{1, 3-i}, u_{2,3-i}\}\cup U_i$ induces a good copy of $K_2(\mathbf{3})$ in $G$, a contradiction.

Summarising our observations in the paragraph above, every $y\in Y\setminus U$ sends either (a) exactly three edges of multiplicity $a+1$ and at least (in fact exactly) one edge of multiplicity $a-1$ into $U$, or (b) at most two edges of multiplicity $a+1$ into $U$. Since $(a-1)(a+1)<a^2$, it follows that the contribution to $p_U$ of each $y\in Y\setminus U$ is at most $a\left(\frac{a+1}{a}\right)^{\frac{1}{2}-\frac{1}{6}}< a\left(\frac{a+1}{a}\right)^{\frac{1}{2}}$.

On the other hand, every $z\in Z$ can send at most  four edges of multiplicity $a+1$ into $U$;  indeed suppose this was not the case and $z$ sent at least five edges of multiplicity $a+1$ into $U$. Then one can choose $2$-vertex subsets $U_1'\subset U_1$ and $U_2'\subset U_2$ such that $U'=U_1'\cup U_2'$ induces a good copy of $K_2(\mathbf{2})$ and $z$ sends edges of
multiplicity $a+1$ into all vertices of $U'$,  contradicting Proposition~\ref{prop: basic properties of Y nhoods}~(v).

 The contribution of each $z\in Z$ to $p_U$ is thus at most $a\left(\frac{a+1}{a}\right)^{\frac{2}{3}}$. By geometric averaging, it follows that there exists $u\in U$ with
\begin{align*}
p_{Y\cup Z}(u)& \leq p_U< a^{(1-\alpha)n}\left(\frac{a+1}{a}\right)^{\frac{1}{2}\beta n +\frac{2}{3}(1-\alpha-\beta)n+O(1)},
\end{align*}
whence by Lemma~\ref{lemma: if exists lowdeg vertex in YcupZ then done} $G$ contains a product-poor vertex and we are done.

\textbf{Case 2: $\mathbf{r\geq3}$.} Let $p_U$ and $p_W$ be the geometric-means of the product-degrees $p_{Y\cup Z}(v)$ over $v\in U$ and $v\in W$ respectively. Let us now consider the contribution of $y\in Y\setminus \left(U\cup W\right)$ to $p_U$ and $p_W$.  Note that since $G^{(a+1)}[Y]$ is $K_{r+1}$-free, every such vertex $y$ can send edges of multiplicity $a+1$ to at most $r-1$ of the parts $U_1$, $U_2$, $W_1, \ \ldots ,\ W_{r-2}$.  It follows that each such vertex $y$ must fall within one of the following mutually exclusive types.

%In particular if $y$ sends at least four edges of multiplicity $a+1$ into $U_1\cup U_2$, then there is some part $W_i$ to which $y$ sends no edge of multiplicity $a+1$.

%Since $G[Y]$ does not contain a good $K_r(\mathbf{3})$ and  $G^{(a+1)}[Y]$ is $K_{r+1}$-free, we have that one of the following must hold:
\begin{itemize}
	\item \textbf{Type Y1:} $y$ send edges of multiplicity $a+1$ to all of $W$. Then there is some part $U_i$, $i\in [2]$ to which $y$ sends no edge of multiplicity $a+1$. If all edges from $y$ to this part $U_i$ have multiplicity $a$, then we note that at least one of the edges from $y$ to $U_{3-i}$ must have multiplicity $a-1$, since otherwise $\left(U\setminus \{u_{3, i}\}\right)\cup \{y\}\cup W$ (which is a subset of $Y$) induces a good copy of $K_{r}(\mathbf{3})$, a contradiction.

	Summarising, $y$ sends either (a)  at most two edges of multiplicity $a+1$ into $U$, or (b) exactly three edges of multiplicity $a+1$ and at least (in fact, exactly) one edge of multiplicity $a-1$ into $U$. Since $(a+1)(a-1)<a^2$, it follows that $y$'s contribution to $p_U$ is at most $a\left(\frac{a+1}{a}\right)^{\frac{1}{3}}$, while its contribution to $p_W$ is $a+1$;
	
	\item \textbf{Type Y2:} $y$ sends edges of multiplicity $a+1$ to at least four vertices in $U$ --- and in particular to both parts $U_1$ and $U_2$. As we observed, this implies there is some part $W_i$, $i\in [r-2]$, such that  $y$ sends no edge of multiplicity $a+1$ into $W_i$. In particular $y$ sends at most $3(r-3)$ edges of multiplicity $a+1$ to $W$, whence its contributions to $p_U$ and $p_W$ are at most $(a+1)$ and $a\left(\frac{a+1}{a}\right)^{1-\frac{1}{r-2}}$ respectively;
	
	\item \textbf{Type Y3}: $y$ sends at most three edges of multiplicity $a+1$ into $U$ and at most $3(r-2)-1$ edges of multiplicity $a+1$ into $W$, whence its contributions to $p_U$ and $p_W$ are at most $a\left(\frac{a+1}{a}\right)^{\frac{1}{2}}$ and $a\left(\frac{a+1}{a}\right)^{1-\frac{1}{3(r-2)}}$ respectively.
\end{itemize}
We now turn our attention to $z\in Z$. Recall that by Proposition~\ref{prop: basic properties of Y nhoods}~(iv), all edges from $z$ to $U\cup W\subseteq Y$ have multiplicity at least $a$. We classify $z\in Z$ into two types, as follows.
\begin{itemize}
	\item \textbf{Type Z1:} $z$ sends at most four edges of multiplicity $a+1$ into $U$, whence its contributions to $p_U$ and $p_W$ are at most $a\left(\frac{a+1}{a}\right)^{\frac{2}{3}}$ and $a+1$ respectively;
	\item \textbf{Type Z2}: $z$ sends at least five edges of multiplicity $a+1$ into $U$. Then it is possible to choose size two subsets $U_1'\subset U_1$ and $U_2'\subset U_2$ such that $U_1'\cup U_2'\subseteq N^{(a+1)}(z)$ and at least one of the vertices $u_{1,3}$, $u_{2,3}$ is missing from the $4$-set $U_1'\cup U_2'$.  Now $U_1'\cup U_2'\cup W$ induces a good copy of $K_r(\mathbf{2}^{(2)} \mathbf{3}^{(r-2)})$ in $G[Y]$. By Proposition~\ref{prop: basic properties of Y nhoods}~(v), we know that $z$ can send at most $2r-1$ edges into a good copy of $K_r(\mathbf{2})$ lying inside $G[Y]$. This implies that there is some part $W_i$, $i\in [r-2]$, such that $z$ sends at most one edge of multiplicity $a+1$ into $W_i$. In particular,  $z$ can send at most $3(r-3)+1$ edges of multiplicity $a+1$ into $W$ in total. Its contributions to $p_U$ and $p_W$ are thus at most $a+1$ and $a\left(\frac{a+1}{a}\right)^{1-\frac{2}{3(r-2)}}$ respectively. 
\end{itemize}
For $i\in [3]$ let $\theta_i$ be the proportion of vertices in $Y\setminus\left(U\cup W\right)$ of Type Yi, and let $\phi$ be the proportion of vertices in $Z$ of type Z1. Then, using $\vert Y\vert =\beta n$, $\vert U\cup W\vert =O(1)$ and $\theta_1+\theta_2+\theta_3=1$, the contribution to $p_W$ from vertices $y\in Y$ is at most
\begin{align*}
\left(\prod_{y\in Y}p_W(y)\right)^{\frac{1}{\vert W\vert }}&< (a+1)^{\vert U\cup W\vert} \left(\prod_{y\in Y\setminus \left(U\cup W\right)}p_W(y)\right)^{\frac{1}{3(r-2)}}\\
& \leq a^{\beta n}\left(\frac{a+1}{a}\right)^{\left(\theta_1+\theta_2\left(1-\frac{1}{r-2}\right)+\theta_3\left(1-\frac{1}{3(r-2)}\right)\right)\beta n+O(1)}= a^{\beta n}\left(\frac{a+1}{a}\right)^{\left(\frac{r-3}{r-2} +\frac{\theta_1+\frac{2}{3}\theta_3}{r-2}\right)\beta n+O(1)}.
\end{align*}
%If $\frac{(r-1)\left(\theta_1+\theta_3\right)-1 }{r-2}\leq \frac{r-2}{r-1}$,
If $\frac{r-3}{r-2} +\frac{\theta_1+\frac{2}{3}\theta_3}{r-2} \leq \frac{r-2}{r-1}$, then by geometric averaging there is some vertex $w\in W$ with $p_Y(w)\leq a\left(\frac{a+1}{a}\right)^{\frac{r-2}{r-1}\beta n+O(1)}$, which by Lemma~\ref{lemma: if a+1 neighbourhoods are sparse, done} implies $G$ contains a product-poor vertex, and so we are done. Rearranging terms, we may thus assume that 
\begin{align}\label{eq: almost lemma, theta1 + theta3 bound}
 (r-1)\left(\theta_1+\frac{2}{3}\theta_3\right) > 1.
\end{align}
Now consider the quantity $p$ given by $p:=\left(p_U\right)^{\frac{2}{r}}\left(p_W\right)^{\frac{r-2}{r}}$. Substituting the upper bounds we derived on the contributions of vertices of Types Y1--Y3 and Z1--Z2 to $p_U$ and $p_W$, we see that
\begin{align*}
&\frac{p}{a^{\vert Y\cup Z\vert}} =\left(\frac{p_U}{a^{\vert Y\cup Z\vert}}\right)^{\frac{2}{r}}\left(\frac{p_W}{a^{\vert Y\cup Z\vert}}\right)^{\frac{r-2}{r}}\\
&<\left(\frac{a+1}{a}\right)^{\vert U \cup W\vert + \vert Y\setminus (U\cup W)\vert \left(\frac{2}{r}\left(\frac{\theta_1}{3}+ \theta_2+ \frac{\theta_3}{2}\right)+\frac{r-2}{r}\left(\theta_1+\frac{\theta_2(r-3)}{r-2}+ \frac{\theta_3(3(r-2)-1)}{3(r-2)}\right)\right) + \vert Z\vert \left(\frac{2}{r}\left(\frac{2\phi}{3}+ 1-\phi\right)+ \frac{r-2}{r}\left(\phi + \frac{(1-\phi)(3(r-2)-2}{3(r-2)}\right)\right)}\\
&=\left(\frac{a+1}{a}\right)^{\left(\frac{r-1}{r}-\frac{(\theta_1+\theta_3)}{3r}\right)\beta n +\left(\frac{r-1}{r}+\frac{1}{3r}\right)(1-\alpha -\beta)n+O(1)}= \left(\frac{a+1}{a}\right)^{\frac{r-1}{r}(1-\alpha )n-\frac{1}{3r}\left(\beta(\theta_1+\theta_3) -(1-\alpha-\beta)\right)n +O(1)}.
\end{align*} 
Since $\beta >\frac{r-1}{r}$ by~\eqref{eq: beta bound} and since $(r-1)(\theta_1+\theta_3)>1$ by~\eqref{eq: almost lemma, theta1 + theta3 bound}, $p$ is at most $a^{(1-\alpha )n}\left(\frac{a+1}{a}\right)^{\frac{r-1}{r}(1-\alpha)n +O(1)}$. By weighted geometric averaging, some vertex $v\in U\cup W$ satisfies $p_{Y\cup Z}(v)\leq p$, whence  $G$ contains a product-poor vertex by Lemma~\ref{lemma: if exists lowdeg vertex in YcupZ then done}, and we are done.
\end{proof}
With Lemma~\ref{lemma: almost good is good enough} in hand, we shift our perspective slightly. Recall  we had shown in Lemma~\ref{lemma: t geq r-1} that $G[Y]$ contains a good $K_{r}(\mathbf{2}^{(1)}\mathbf{3}^{(r-1)})$. It follows that there is a $3(r-1)$-set $W$ in $G$ such that $G[W]$ induces a good $K_{r-1}(\mathbf{3})$ and such that the joint neighbourhood \[N_W:=\bigcap_{w\in W} N^{(a+1)}(w)\]
contains a good $K_2(1,2)$ (with $x$ corresponding to the part of size $1$). Note that $W$ and $N_W$ are disjoint (since $G[W]$ contains edges of multiplicity $a$).

We now prove three lemmas about $N_W$ to conclude our proof. All of these will be proved by weighted geometric averaging arguments reminiscent of those used in Lemma~\ref{lemma: if contain good Kr structure, then done}. Let $C_5$ denote the $5$-cycle and $P_4$ denote the path on $4$ vertices (i.e. the graph obtained from $C_5$ by deleting one of the vertices). 
\begin{lemma}\label{lemma: good C5 in NW}
If $N_W$ contains a good $C_5$, then $G$ contains a product-poor vertex.	
\end{lemma}
\begin{proof}
By Lemmas~\ref{lemma: if contain good Kr structure, then done} and~\ref{lemma: almost good is good enough}, we may assume that for every vertex $v\in V$, the $(a+1)$-neighbourhood of $v$ contain no good or almost good $K_r(\mathbf{3})$. Let $C$ denote the vertex-set of a good $C_5$ in $N_W$. Consider the quantity 	
\[p:= \left(\prod_{v\in W}p(v)\right)^{\frac{1-x_{\star}(r+1,1)}{3r}} \left(\prod_{v\in C}p(v)\right)^{\frac{1+(r-1)x_{\star}(r+1,1)}{5r}}.\]
Observe that $p$ is just a weighted geometric mean of the product-degrees of the vertices in $C\cup W$. Consider now a vertex $v\in V\setminus \left( C\cup W\right)$. We have three cases to consider.

\textbf{Case 1.} If $v$ sends at most $3(r-2)$ edges of multiplicity $a+1$ into $W$, then the contribution of $v$ to $p$ is at most 
%\begin{align*}
$a\left(\frac{a+1}{a}\right)^{\frac{r-1+x_{\star}(r+1,1)}{r}}$.
%\end{align*}

\textbf{Case 2.} If on the other hand $v$ sends exactly $3(r-2)+1$ edges of multiplicity $a+1$ into $W$, then, since $G$ contains no good $K_{r+2}$, we have that $v$ can send at most two edges of multiplicity $a+1$ into $C$ (indeed otherwise $v$ sends an edge of multiplicity $a+1$ to each of the parts of $W$ and to both ends of an edge in $N_W$). The contribution of $v$ to $p$ is thus at most
\begin{align*}
a\left(\frac{a+1}{a}\right)^{1-\frac{2(1-x_{\star}(r+1,1))}{3r}-\frac{3(1+(r-1)x_{\star}(r+1,1))}{5r}}.
\end{align*}
Since
\begin{align*}
\frac{2(1-x_{\star}(r+1,1))}{3r}+\frac{3(1+(r-1)x_{\star}(r+1,1))}{5r}-\frac{1-x_{\star}(r+1,1)}{r}=\frac{4+(9r-4)x_{\star}(r+1,1)}{15r}>0,
\end{align*}
it follows that $v$ contributes (strictly) less than $a\left(\frac{a+1}{a}\right)^{\frac{r-1+x_{\star}(r+1,1)}{r}}$ to $p$.

\textbf{Case 3.} Finally if  $v$ sends at least $3(r-2)+2$ edges of multiplicity $a+1$ into $W$, then it cannot send any edge of multiplicity $a-1$ into $W$ (recall that vertices joined by such edges must be clones of each other, and observe that $v$ cannot be the clone of any vertex in $W$ as all vertices in $W$ send exactly $3(r-2)$ edges of multiplicity $a+1$ into $W$).

 Suppose $v$ sends an edge of multiplicity $a+1$ into some vertex $c\in C$. If $v$ sends an edge of multiplicity $a+1$ into both vertices of $N^{(a+1)}(c)\cap C$, then $G^{(a+1)}[W\cup C\cup\{v\}]$ contains a copy of $K_{r+2}$, contradicting Proposition~\ref{prop: basic properties of Y nhoods}(i). Thus $v$ sends at least one one edge of multiplicity at most $a$ into $N^{(a+1)}(c)\cap C$; if it sends edges of multiplicity at least $a$ into both vertices of  $N^{(a+1)}(c)\cap C$, then $W\cup \{v\}\cup C$ contains a good or almost good $K_r(\mathbf{3})$, again a contradiction. So $v$ must send an edge of multiplicity $a-1$ to one of the vertices in $N^{(a+1)}(c)\cap C$. Since $(a+1)^2(a-1)>a^3$ for $a\geq 3$ this gives a larger contribution to $p$ than if $v$ sent only edges of multiplicity $a$ into $C$. Using also the inequality $(a+1)^2(a-1)<(a+1)a^2$, we can thus upper-bound the contribution of $v$ to $p$ by
	\begin{align*}
	a\left(\frac{a+1}{a}\right)^{1-\frac{4}{5r}\left(1+(r-1)x_{\star}(r+1,1)\right)}.
	\end{align*}	
Now for all $a\geq 2$, we have $(a+1)^3(a-1)>a^3$, and thus
\begin{align*}
\frac{4}{5r}\left(1+(r-1)x_{\star}(r+1,1)\right)-\frac{1-x_{\star}(r+1,1)}{r}=\frac{1}{5r}\frac{\log\left(\frac{(a+1)^{3r}(a-1)^r}{a^{4r}}\right)}{\log\left(\frac{(a+1)^{r+1}}{(a-1)^ra}\right)}>0,
\end{align*}
whence in this last case again $v$ contributes at most $a\left(\frac{a+1}{a}\right)^{\frac{r-1+x_{\star}(r+1,1)}{r}}$ to $p$.

It follows from our case analysis  that $p\leq a^{n}\left(\frac{a+1}{a}\right)^{\frac{r-1+x_{\star}(r+1,1)}{r}n+O(1)}$, whence by geometric averaging one of the vertices in $C\cup W$ is product-poor, and we are done.
\end{proof}

\begin{lemma}\label{lemma: good P4 in NW}
If $N_W$ contains a good $P_4$, then $G$ contains a product-poor vertex.
\end{lemma}
\begin{proof}
By Lemmas~\ref{lemma: if contain good Kr structure, then done} and~\ref{lemma: almost good is good enough}, we may assume that for every vertex $v\in V$, the $(a+1)$-neighbourhood of $v$ contains no good or almost good $K_r(\mathbf{3})$. Further by Lemma~\ref{lemma: good C5 in NW} we may assume that $N_W$ does not contain a good $C_5$. 

Let $U=\{u_1,u_2, u_3, u_4\}$ induce a $P_4$ in $N_W$, with $u_1$ and $u_4$ being the ends of the path, and  $u_2, u_3$ the two middle vertices. Consider the quantity
\[p:= \left(\prod_{v\in W}p(w)\right)^{\frac{1-x_{\star}(r+1,1)}{3r}} \Bigl(p(u_1)p(u_4)\Bigr)^{\frac{1+(r-1)x_{\star}(r+1,1)}{6r}} \Bigl(p(u_2)p(u_3)\Bigr)^{\frac{1+(r-1)x_{\star}(r+1,1)}{3r}},\]
which is a weighted geometric mean of the product-degrees of the vertices in $U\cup W$.

Consider a vertex $v\in V\setminus \left( C\cup W\right)$. We have four cases to consider.

\textbf{Case 1.} If $v$ sends at most $3(r-2)$ edges of multiplicity $a+1$ into $W$, then the contribution of $v$ to $p$ is at most 
$a\left(\frac{a+1}{a}\right)^{\frac{r-1+x_{\star}(r+1,1)}{r}}$.

\textbf{Case 2.} If $v$ sends exactly $3(r-2)+1$ edges of multiplicity $a+1$ into $W$, then, as $G$ is $K_{r+2}$-free, $v$ cannot send edges of multiplicity $a+1$ to both ends of an edge in $N_W$. In particular, the set of vertices in $U$ it sends edges of multiplicity $a+1$ to must be a subset of one of the pairs $\{u_1, u_3\}$, $\{u_2, u_4\}$, $\{u_1, u_4\}$. It follows from this that the contribution of $v$ to $p$ is at most the maximum contribution recorded in Case 1 multiplied by a factor of
\begin{align*}
\left(\frac{a+1}{a}\right)^{\frac{1-x_{\star}(r+1,1)}{3r}} \left(\frac{a+1}{a}\right)^{-\frac{1+(r-1)x_{\star}(r+1,1)}{2r}}<1.
\end{align*}

 \textbf{Case 3.} If $v$ sends exactly $3(r-2)+2$ edges of multiplicity $a+1$ into $W$, then the last edge it sends into $W$ must have multiplicity $a$ (it cannot be $a-1$, since $v$ clearly cannot be the clone of a vertex in $W$: vertices in $W$ only send $3(r-2)$ edges of multiplicity $a+1$ into $W$).

  Suppose $v$ sends an edge of multiplicity $a+1$ into one of the middle vertices $\{u_2, u_3\}$ of $U$, say $u_2$. Then $v$ must send an edge of multiplicity $a-1$ to one of $u_2$'s neighbours $u_1$ and $u_3$, as otherwise $N^{(a+1)}(u_2)$ contains an almost good $K_r(\mathbf{3})$.   If $vu_1$ has multiplicity $a-1$, then $v$ is a clone of $u_1$ and hence we have 
  \begin{align*}w(vu_1)=a-1, && w(vu_2)= a+1, && w(vu_3)=a, && w(vu_4)=a.
  \end{align*} 
  On the other hand if $vu_3$ has multiplicity $a-1$, then $v$ is a clone of $u_3$ and   \begin{align*}w(vu_1)=a, && w(vu_2)= a+1, && w(vu_3)=a-1, && w(vu_4)=a+1.
  \end{align*} 
  Finally if $v$ fails to send any edge of multiplicity $a+1$ into $\{u_2, u_3\}$, then
  \begin{align*}w(vu_1)\leq a+1, && w(vu_2)\leq a, && w(vu_3)\leq a, && w(vu_4)\leq a+1.
  \end{align*} 
  Plugging these three different bounds on the multiplicities of edges from $v$ to $U$ into the definition of $p$, we see that in Case 3, the contribution of $v$ to $p$ is at most that recorded in Case 1 multiplied by a factor of
  \begin{align*}
  \left(\frac{a+1}{a}\right)^{\frac{2\left(1-x_{\star}(r+1,1)\right)}{3r}} \left(\frac{a+1}{a}\right)^{-\frac{2\left(1+(r-1)x_{\star}(r+1,1)\right)}{3r}}<1,
  \end{align*}
 attained if $w(vu_1)=w(vu_4)=a+1$ and $w(vu_2)=w(vu_3)=a$.

 \textbf{Case 4.} If $v$ sends edges of multiplicity $a+1$ to all $3(r-1)$ vertices in $W$, then we have two possibilities to consider.

 If $v$ sends an edge of multiplicity $a+1$ into one of the middle vertices of $U$, say $u_2$, then, as in the Case 3, it must send an edge of multiplicity $a-1$ into one of $\{u_1, u_3\}$, so that we have
 \begin{align}\label{eq: case 4 prodbound1}
 \left(w(vu_1)w(vu_4)\right)\cdot\left(w(vu_2)w(vu_3)\right)^2=(a-1)a^3(a+1)^2&&\textrm{ or } && (a-1)^2a(a+1)^3.
 \end{align}
 On the other hand, suppose $v$ does not send an edge of multiplicity $a+1$ into the middle vertices $\{u_2, u_3\}$ of $U$. Then both $vu_2$ and $vu_3$ must have multiplicity exactly $a$ --- indeed otherwise $v$ would have to send an edge of multiplicity $a-1$ to one of $\{u_2, u_3\}$, say $u_2$, which would imply $v$ is a clone of $u_2$ and thus sends an edge of multiplicity $a+1$ to $u_3$, a contradiction. Since $N_W$ does not contain a good $C_5$, this implies that $v$ can send an edge of multiplicity $a+1$ into at most one of the end-vertices $\{u_1, u_4\}$ and
 \begin{align}\label{eq: case 4 prodbound2}
 \left(w(vu_1)w(vu_4)\right)\cdot\left(w(vu_2)w(vu_3)\right)^2\leq a^5(a+1).
 \end{align} 
Using our bounds~\eqref{eq: case 4 prodbound1} and~\eqref{eq: case 4 prodbound2} on the contribution to $p$ of edges from $v$ to $U$ and the fact that $a^5(a+1)>\max\left\{(a-1)a^3(a+1),  (a-1)^2a(a+1)^3\right\}$, we see that $v$'s contribution to $p$ is at most that recorded in Case 1 multiplied by a factor of
  \begin{align*}
 \left(\frac{a+1}{a}\right)^{\frac{\left(1-x_{\star}(r+1,1)\right)}{r}} \left(\frac{a+1}{a}\right)^{-\frac{5\left(1+(r-1)x_{\star}(r+1,1)\right)}{6r}}=\left(\frac{a+1}{a}\right)^{\frac{1-(5r+1)x_{\star}(r+1,1)}{6r}}<1,
 \end{align*}
with the last inequality following from the fact that for all $a\geq 2$, $(a+1)^4(a-1)>a^5$ and hence
\begin{align*}
(5r+1)x_{\star}(r+1,1)-1=\frac{\log\left(\frac{(a+1)^{4r}(a-1)^r}{a^{5r}}\right)}{\log\left(\frac{(a+1)^{r+1}}{(a-1)^ra}\right)}>0.
\end{align*}
 Since in each of  Cases 1--4 the contribution to $p$ is at most that recorded in Case 1, we get that 
 \begin{align*}
 p\leq  a^n \left(\frac{a+1}{a}\right)^{\frac{r-1+x_{\star}(r+1,1)}{r}n+o(n)}.
 \end{align*}
 By geometric averaging, it follows that one of the vertices in $U\cup W$ is product-poor, and we are done.
\end{proof}

\begin{lemma}\label{lemma: exists good P4 in NW}
	Either $N_W$ contains a good $P_4$ or $G$ contains a product-poor vertex.
\end{lemma}
\begin{proof}	
As we have shown, either $G$ contains a product-poor vertex or $N_W$ contains a good $K_2(1,2)$. Let $\{x\}$ and $U=\{u_1, u_2\}$ be the vertex-sets corresponding to the two parts in this $K_2(1,2)$. Suppose that $N_W$ does not contain a good $P_4$.  We shall show this implies $G$ contains a product-poor vertex.

Note that by Lemmas~\ref{lemma: if contain good Kr structure, then done} and~\ref{lemma: almost good is good enough} we may assume that the $(a+1)$-neighbourhood of $x$ (indeed, of any vertex) does not contain a good or almost good $K_r(\mathbf{3})$.

Let $p_W$ denote the geometric mean of the product-degrees of the vertices from $W$ and set $p_U:=\sqrt{p(u_1)p(u_2)}$. Consider the quantity
\[p:=\Bigl(p(x)\Bigr)^{x_{\star}(r+1,1)}\Bigl(p_U\Bigr)^{\frac{1-x_{\star}(r+1,1)}{r}}\Bigl(p_W\Bigr)^{\frac{(r-1)(1-x_{\star}(r+1,1))}{r}}, \]
which is a weighted geometric mean of the product-degrees of the vertices in $\{x\}\cup U\cup W$.  Much as in the proof of Lemma~\ref{lemma: if contain good Kr structure, then done} we shall show that $p$ cannot be too large. Indeed, consider a vertex $v \in V\setminus \left(\{x\}\cup U\cup W\right)$.

\textbf{Case 1.} If $v$ sends an edge of multiplicity $a-1$ to $x$, then its contribution to $p$ is exactly \begin{align*}
(a-1)^{x_{\star}(r+1,1)}(a+1)^{1-x_{\star}(r+1,1)}=a\left(\frac{a+1}{a}\right)^{\frac{r-1+x_{\star}(r+1,1)}{r}}.
\end{align*}

\textbf{Case 2.} If $v$ sends an edge of multiplicity $a+1$ to $x$, then it can send edges of multiplicity $a+1$ to at most $r-1$ of the parts of the good $K_r(\mathbf{2}^{(1)}\mathbf{3}^{(r-1)})$ induced by $U\cup W$ (for otherwise we would have a good $K_{r+2}$ in $G$). Thus its contribution to $p$ is at most
\begin{align*}
a\left(\frac{a+1}{a}\right)^{1-\frac{1-x_{\star}(r+1,1)}{r-1}} = a\left(\frac{a+1}{a}\right)^{\frac{r-1+x_{\star}(r+1,1)}{r}}.
\end{align*}

\textbf{Case 3.} If $v$ sends an edge of multiplicity $a$ to $x$, then we claim it sends either (a) an edge of multiplicity $a-1$, or (b) at least $2$ edges of multiplicity $a$ into $U\cup W$.

Indeed, suppose neither of these occurs, i.e. that all edges from $v$ to $U\cup W$ have multiplicity at least $a$, and that all but at most one have multiplicity $a+1$. If all these edges have multiplicity $a+1$, then $G^{(a+1)}[\{x,v\}\cup U\cup W]$ contains a copy of $K_{r+2}$, contradicting Proposition~\ref{prop: basic properties of Y nhoods} part (i). Thus we may assuyme that $v$ sends exactly one edge of multiplicity $a$ into $U\cup W$. If this edge of multiplicity $a$ is to a vertex in $U$, then $U\cup\{x,v\}$ induces a good $P_4$ in $N_W$, a contradiction.  On the other hand if the edge of multiplicity $a$ is to a vertex in $W$, then $\{x,v\}\cup U\cup W$ contains a good $K_{r+1}(\mathbf{2})$, contradicting $G\in \mathcal{H}(n,2r+2, \Sigma_{r+1,1}(a, 2r+2))$.

It readily follows that the contribution of $v$ to $p$ is at most
\begin{align*}
a \left(\frac{a+1}{a}\right)^{1-x_{\star}(r+1,1)-\frac{2(1-x_{\star}(r+1,r))}{3r}}& <a \left(\frac{a+1}{a}\right)^{\frac{r -1 +x_{\star}(r+1,r))}{r}},
\end{align*}
with the inequality following from the fact that
\begin{align*}
\frac{1}{3r} \left(1-x_{\star}(r+1,1)\right)- x_{\star}(r+1, 1)=-\frac{1}{3r}\frac{\log\left(\frac{(a+1)^{2r}(a-1)^r}{a^{3r}}\right)}{\log\left(\frac{(a+1)^{r+1}}{(a-1)^ra}\right)}<0
\end{align*}
for all $a\geq 2$.

 Since in every case the contribution to $p$ is at most $a \left(\frac{a+1}{a}\right)^{\frac{r -1 +x_{\star}(r+1,r))}{r}}$, it follows that
\begin{align*}
p\leq  a^n \left(\frac{a+1}{a}\right)^{\frac{r-1+x_{\star}(r+1,1)}{r}n+o(n)}.
\end{align*}
By geometric averaging, one of the vertices in $\{x\}\cup U\cup W$ is product-poor, and we are done.	
\end{proof}

Combining Lemmas~\ref{lemma: t>0}, \ref{lemma: t geq r-1}, \ref{lemma: exists good P4 in NW} and \ref{lemma: good P4 in NW} we see that irrespective of the value of $t$, $G$ must contain a product-poor vertex.  This concludes the proof of the inductive step.
\end{proof}

\section{Further questions and conjectures}\label{section: other questions and conjectures}
\noindent There is much work yet to be done on the Mubayi--Terry problem. We discuss below some of the more promising directions we see for future research.

\textbf{Other cases of of Conjecture~\ref{conjecture: main conjecture}.} The most obvious open problem is that of the remaining cases of Conjecture~\ref{conjecture: main conjecture}. With the techniques developed in this paper, we suspect resolving the $d=2$ case (by working out how to handle a broader range of possible edge multiplicities in putative extremal constructions) could lead to a resolution of the full conjecture. One possible path towards this would be some appropriate refinement of Lemma~\ref{lemma: if exists lowdeg vertex in YcupZ then done} which takes into account the fact that $Z$ is split up into the $d-1$ sets $Z_i$, $i=0, 1, \ldots, d-1$ with $Z_i=N^{(a-i)}(x)$. The special case $(s,q)= (5, \binom{5}{2}a+4)$ with $a\geq 3$ is the smallest open case, and would provide a good testing ground for such refinements.

\textbf{Stability and exact values.} We strongly believe that the equality $\mathrm{ex}_{\Pi}(n, 2r, \Sigma_{r,1}(a, 2r))=\Pi_{r,1}(a, n)$ holds for all $a, r\geq 2$ and all $n$ sufficiently large, so that the asymptotic equality we established in Theorem~\ref{theorem: d=1 asymptotic} is not the last word even in the case of Conjecture~\ref{conjecture: main conjecture} treated in this paper. A natural step towards such an exact result would be to obtain a stability result for Theorem~\ref{theorem: d=1 asymptotic} showing almost product-extremal $G$ in $\mathcal{F}(n, 2r, \Sigma_{r,1}(a, 2r))$ must lie close in edit distance to product-extremal graphs from $\mathcal{T}_{r,1}(a,n)$.

We have not attempted to prove such a result in this paper, which is already overly long and technical. However we suspect a partial stability result can be extracted from our proof.  Indeed, one can show by a simple vertex-removal argument that an almost product-extremal $G'$ in $\mathcal{H}(n, 2r, \Sigma_{r,1}(a, 2r))$ must contain at most $o(n)$ strictly product-poor vertices. The case analyses in the proofs of our Lemmas in Section~\ref{section: finding r-partite} then imply that all but $o(n)$ vertices must fall within one given type (since usually only a small subset of the types give an optimal contribution to the various $p$, $p_u$ and $p_W$ quantities we consider in our averaging arguments, while the other types give strictly worse contributions). Such information could be used to characterise the large-scale structure of $G$. The main challenge would be then to show that almost product-extremal multigraphs $G$ from the larger family $\mathcal{F}(n, 2r, \Sigma_{r,1}(a, 2r))$ lie close in edit distance to some multigraph  $G'\in \mathcal{F}(n, 2r, \Sigma_{r,1}(a, 2r))$: it is not immediately obvious how to obtain such a stability version of Proposition~\ref{prop: reducing to good graphs with clique decompositions} .

\textbf{Other values of $\mathbf{(s,q)}$.} The next most obvious open problem is to resolve what happens for $(s,q)$ when $q$ is not of the form $q=\Sigma_{r,d}(a,s)$ for some  integers $r\geq 1$ and $a>d\geq 0$. In~\cite[Construction 12.1]{DayFalgasRavryTreglown20+}, Day, Falgas-Ravry and Treglown considered `iterated' versions of Construction~\ref{construction: lower bound} --- multigraphs obtained by taking a graph from $\mathcal{T}_{r,d}(a,n)$, replacing the special part $V_0$ by a multigraph from $\mathcal{T}_{r', d'}(a-d, \vert V_0\vert)$, and repeating this procedure. This gives lower-bound constructions for additional pairs $(s,q)$ not covered by Conjecture~\ref{conjecture: main conjecture}. The authors of ~\cite{DayFalgasRavryTreglown20+} asked whether these were asymptotically tight. Given the previous work of F\"uredi and K\"undgen (in which similar iterated constructions appear, albeit with very different and much simpler relative part sizes) and the work in the present paper, it is tempting to guess that the answer to this question might be affirmative. The smallest test case of this may be $(s,q)=(6, \binom{6}{2}a+8)$ for $a\geq 2$.

\textbf{Asymptotically flat intervals.}	Even with the iterated constructions above and $a$ large, there are still values of $(s,q)$ which do not have their ``own'' lower-bound constructions, but only constructions that are also valid for $(s,q-1)$. We think these pairs may correspond to intervals in which the value of $\mathrm{ex}_{\Pi}(s,q)$ does not change as we change the value of $q$ (keeping $s$ fixed).

Our intuition is based on the special case $q_0=\Sigma_{r,0}(a,s)$: if we want to make any non-trivial increase to the asymptotic product of the edge multiplicities, and we restrict ourselves to `iterated' versions of Construction~\ref{construction: lower bound}, we must introduce at least one new part to our construction, which in turn suggests the maximum edge-sum over $s$-sets must increase by at least $\left\lfloor \frac{s-1}{r}\right\rfloor$. Thus for any $q$ with $q_0\leq q< q_0 + \lfloor \frac{s-1}{r}\rfloor$ there should be no other asymptotically different constructions available than those from $\mathcal{T}_{r,0}(a,n)$. Formally, this yields:	
\begin{conjecture}\label{conjecture: non-jumps}
For every $r,a\in \mathbb{N}$ and for every $s\geq 2r+1$, we have
\[\mathrm{ex}_{\Pi}(s, \Sigma_{r,0}(a,s))=\mathrm{ex}_{\Pi}(s, \Sigma_{r,0}(a,s)+1)=\ldots   =\mathrm{ex}_{\Pi}\left(s, \Sigma_{r,0}(a,s)+\Bigl\lfloor \frac{s-1}{r}\Bigr\rfloor-1\right)= a^{\frac{1}{r}}(a+1)^{\frac{r-1}{r}}.\]
\end{conjecture}
Mubayi and Terry showed in~\cite[Theorem 3]{MubayiTerry20} that Conjecture~\ref{conjecture: non-jumps} holds for $r=1$. Revisiting their work may provide a path towards proving Conjecture~\ref{conjecture: non-jumps}. Beyond that, there may be other asymptotically flat intervals where the value of $\mathrm{ex}_{\Pi}(s,q)$ does not change as we increase $q$ --- the meta-conjecture should perhaps be that for any $s$ and all $q$ large enough, $\mathrm{ex}_{\Pi}(s,q)$ is the maximum of the asymptotic product density $P(G)^{1/\binom{n}{2}}$ over $G$ belong to the collection of `iterated versions of Construction~\ref{construction: lower bound} on $n$ vertices with the $(s,q)$-property', but we are currently quite far from a position in which we could confidently put forward such a statement.

\textbf{Reducing to the base case in Conjecture~\ref{conjecture: main conjecture}.} Finally, it would be nice to improve~\cite[Theorem 3.11]{DayFalgasRavryTreglown20+} by getting rid of the ``$a$ sufficiently large'' condition, so that showing $\mathrm{ex}_{\Pi}(n,s, \Sigma_{r,d}(a,s))=\Pi_{r,d}(a,n)^{1+o(1)}$ holds for $s=(r-1)(d+1)+2$ and $a=d+1$ ensures it holds for all $s\geq (r-1)(d+1)+2$ and $a\geq d+1$ (if this statement is true!).

\section*{Acknowledgements}
%The author would like to express his gratitude to an anonymous referee whose careful work helped improve the correctness of the paper --- and in particular for their spotting a subtle mistake in the original formulation of the proof of Theorem 1.6.
The author would like to express his gratitude to two anonymous referees whose careful work helped greatly improve the clarity and correctness of the paper, and in particular for their pointing out two subtle mistakes in the original argument.

\end{document}